\documentclass{amsart}
\usepackage{amsmath,amsfonts,amssymb,amscd,amsthm,epsf}
\newtheorem{prop}{Proposition}[section]
\newtheorem{thm}[prop]{Theorem}
\newtheorem{cor}[prop]{Corollary}
\newtheorem{ques}[prop]{Question}

\newtheorem{conj}[prop]{Conjecture}

\theoremstyle{definition}
\newtheorem{de}[prop]{Definition}
\newtheorem{example}[prop]{Example}

\theoremstyle{remark}
\newtheorem*{remark}{Remark}            
\def\complex{{\mathbb C}}
\def\C{{\mathbb C}}
\def\CP{{\mathbb C \mathbb P}}
\def\zed{{\mathbb Z}}
\def\Z{{\mathbb Z}}
\def\real{{\mathbb R}}
\def\R{{\mathbb R}}
\def\int{\mathop{\rm int}\nolimits}
\def\cl{\mathop{\rm cl}\nolimits}

\def\dim{\mathop{\rm dim}\nolimits}

\def\max{\mathop{\rm max}\nolimits}
\def\id{\mathop{\rm id}\nolimits}

\def\U{\mathop{\rm U}\nolimits}
\def\SO{\mathop{\rm SO}\nolimits}

\def\co{\colon\thinspace}

\begin{document}
\title{Smooth embeddings with Stein surface images}
\author{Robert E. Gompf}
\thanks{Partially supported by NSF grants DMS-0603958 and 1005304.}
\address{Department of Mathematics, The University of Texas at Austin,
1 University Station C1200, Austin, TX 78712-0257}
\email{gompf@math.utexas.edu}
\date{January 8, 2013.}
\begin{abstract}
A simple characterization is given of open subsets of a complex surface that smoothly perturb to Stein open subsets. As applications, $\C^2$ contains domains of holomorphy (Stein open subsets) that are exotic $\R^4$'s, and others homotopy equivalent to $S^2$ but cut out by smooth, compact 3-manifolds. Pseudoconvex embeddings of Brieskorn spheres and other 3-manifolds into complex surfaces are constructed, as are pseudoconcave holomorphic fillings (with disagreeing contact and boundary orientations). Pseudoconcave complex structures on Milnor fibers are found. A byproduct of this construction is a simple polynomial expression for the signature of the $(p,q,npq-1)$ Milnor fiber. Akbulut corks in complex surfaces can always be chosen to be pseudoconvex or pseudoconcave submanifods. The main theorem is expressed via Stein handlebodies (possibly infinite), which are defined holomorphically in all dimensions by extending Stein theory to manifolds with noncompact boundary.
\end{abstract}
\maketitle


\section{Introduction}

There has been much research in the past century devoted to finding {\em Stein manifolds}. In light of such extensive research, many equivalent definitions of these have arisen, perhaps the most succinct being that a Stein $n$-manifold is a complex $n$-manifold (with real dimension $2n$) that admits a proper, holomorphic embedding into some complex Euclidean space $\C^N$. A fundamental theorem of Eliashberg \cite{E} (see also \cite{CE}) gives a complete topological characterization of those smooth $2n$-manifolds that admit Stein structures. The statement is quite simple except when $n=2$, the case of {\em Stein surfaces}. In that case, the characterization is challenging to apply in practice, although it can be simplified by passing to the topological ($C^0$) category to allow exotic smooth structures \cite{Ann}. The success of abstractly characterizing Stein manifolds leads to a more subtle ambient problem: Every open subset $U\subset X$ of a complex manifold inherits a complex structure --- when does this make $U$ into a Stein manifold? (When $X$ is itself a Stein manifold, e.g.\ $\C^n$, such a Stein $U$ is classically called a {\em domain of holomorphy}.) Since the answer can change under $C^0$-small perturbations of $U$, we pose a more topological version of the problem: When is $U$ {\em isotopic} to a Stein manifold? That is, when is the inclusion of $U$  smoothly homotopic through embeddings to one whose image is a Stein open subset? (In this paper, an embedding is a diffeomorphism onto its image, not necessarily proper.) Eliashberg's method can again answer this question, although the $n=2$ case is again more difficult. A broad exposition of both the abstract and ambient theories is given in \cite{yfest}. When $n=2$, the ambient theory again has a topological theorem bypassing the main difficulty of the smooth version by allowing topological isotopies that may change the diffeomorphism type of $U$; this will be given in \cite{steintop} (see also \cite{JSG}). The applications of this are powerful, but have the drawback that the resulting Stein surfaces typically are not diffeomorphic to the interior of any compact manifold with boundary.  In the present paper, we instead require everything to be smooth, so that we can control diffeomorphism types and obtain Stein surfaces $U\subset X$ with smooth, compact, pseudoconvex boundaries. As applications, we construct Stein embeddings of exotic $\R^4$'s in $\C^2$, as well as various Stein embeddings of compact 4-manifolds and pseudoconvex embeddings of 3-manifolds into complex surfaces. When these complex surfaces are compact, this allows the study of {\em concave} holomorphic fillings of 3-manifolds.

The main principle of this paper is that the only obstruction to isotoping $U\subset X$ to a Stein open subset is that of abstractly making $U$ Stein. That is, the ambient problem reduces to the abstract problem, which has been solved by Eliashberg's Theorem. As discussed in \cite{yfest}, this is true in all dimensions, and for $n=2$ is also true in the topological setting \cite{JSG}, \cite{steintop}. The present paper is devoted to the statement, proof and applications of the smooth version with $n=2$.  The main theorem, Theorem~\ref{main}, was announced in \cite{yfest}, but is now cleaner and has a much simpler proof. We ultimately present the theorem with enough generality to apply to \cite{steintop}, but we focus on simpler applications. For the moment, we give a version of the theorem that is both easy to state and fairly general. This version now also follows more systematically, in all dimensions, from \cite{CE} (Theorem~13.8 of that reference, with $J$ inherited from the complex surface, and Weinstein structure constructed from the given Stein structure).

\begin{thm} \label{smooth}
An open subset $U$ of a complex surface is smoothly isotopic to a Stein open subset if and only if the induced complex structure on $U$ is homotopic (through almost-complex structures on $U$) to a Stein structure on $U$.
\end{thm}

We prove this in Section~\ref{Smooth}; it follows either from Theorem~\ref{main} or from a simplified version of its proof. As an immediate application, we have:

\begin{cor} \label{R4}
There are uncountably many diffeomorphism types of exotic $\R^4$'s realized as domains of holomorphy in $\C^2$. That is, there are uncountably many Stein open subsets of $\C^2$ that are homeomorphic to $\real^4$ but pairwise nondiffeomorphic. These can all be found inside a preassigned open ball in $\C^2$, and so arise inside any complex surface.
\end{cor}

\noindent Exotic $\R^4$'s (manifolds homeomorphic but not diffeomorphic to $\R^4$) have a complex history, tracing back to Casson \cite{C}. DeMichelis and Freedman \cite{DF} first showed that uncountably many diffeomorphism types of  exotic $\real^4$'s can be smoothly embedded in the standard $\real^4$. After simplification \cite{BG} (see also \cite{GS}), uncountably many of these were shown to admit Stein structures in \cite{Ann}, but it was previously unknown whether any domain of holomorphy in $\complex^2$ could be an exotic $\real^4$.

\begin{proof}
 Let $R$ be a Stein exotic $\real^4$ from \cite{Ann}. This smoothly embeds in $\real^4$, hence in an open ball $V$ in $\complex^2$, preserving orientation. Since $R$ is contractible, it has only one almost-complex structure, up to homotopy, compatible with the given orientation. Theorem~\ref{smooth} isotopes the embedding in the complex surface $V$ until its image is Stein. While the resulting embedding need not be holomorphic, it is, by definition, a diffeomorphism to its image, which is the required domain of holomorphy. Clearly, we can realize any of the uncountably many diffeomorphism types of exotic $\real^4$'s from \cite{Ann} this way.
\end{proof}

Since most of our applications will involve embeddings of compact manifolds with boundary, we will restate the main principle in the next section, in a form more directly related to the boundary manifolds (Theorem~\ref{main-}). A Stein manifold can equivalently be defined \cite{Gr} as a complex manifold $V$ that admits an {\em exhausting plurisubharmonic function} $\varphi$, which we can assume is a Morse function whose indices are necessarily $\le \dim_\C V$. ``Exhausting'' means the function is proper and bounded below, so without loss of generality it maps to $[0,\infty)$. Plurisubharmonicity (by which we actually mean {\em strict} plurisubharmonicity) is essentially characterized by {\em pseudoconvexity} (actually {\em strict} pseudoconvexity) of the level sets, oriented as the boundaries of the sublevel sets $\varphi^{-1}[0,a]$. This, in turn, implies that the level sets are contact manifolds (positive with respect to the boundary orientation). (We expand on this later; see also \cite{CE}, \cite{OS}.) We use the term {\em Stein domain} for a complex manifold $W$ realizable as the sublevel set of a regular value of such a function on a Stein manifold. Such a $W$ is necessarily compact with pseudoconvex boundary and Stein interior, and the Morse function induces a handlebody structure on it whose handles are nicely compatible with the contact structures and have indices $\le \dim_\C W$. If a Stein manifold admits an exhausting plurisubharmonic Morse function with only finitely many critical points, then after deformation it is the interior of some Stein domain. In complex dimension 2, the handle structure of a Stein domain can be encoded in a {\em Legendrian link diagram}. This does not completely capture the biholomorphism type of the Stein domain, but does capture its diffeomorphism type and almost-complex structure, as well as the contact structures of the level sets. Eliashberg's Theorem in this dimension states that every such Legendrian diagram comes from a Stein domain. (See \cite{Ann} for more details and applications.) Most of our applications involve embedding Stein domains, so we will be finding Stein open subsets that (unlike the above exotic $\R^4$'s) are cut out of complex surfaces by smooth, compact, pseudoconvex 3-manifolds. Theorem~\ref{main-} allows this by expressing the main principle in terms of Stein handlebodies, so that we can look for an embedded Stein domain with a prespecified Legendrian link diagram. Focusing on the boundary, we can look for pseudoconvex embeddings of a suitable prespecified contact 3-manifold. Theorem~\ref{main-} is a simplification of Theorem~\ref{main}, which allows infinite, relative handlebodies. The full generality will be needed in \cite{steintop}, and infinite topology is already needed for Theorem~\ref{smooth}, in particular for exotic $\R^4$'s.

After stating the main principle in terms of handlebodies, we move on to applications. Section~\ref{Apps} discusses embedded Stein domains that have the homotopy type of a point or 2-sphere. We easily construct embeddings of contractible Stein domains in $\C^2$ (Example~\ref{hsphere}). This leads to infinitely many homology 3-spheres with pseudoconvex embeddings in $\C^2$, as boundaries of contractible domains of holomorphy. In contrast, we conjecture that no Brieskorn homology sphere (with either orientation) admits a pseudoconvex embedding in $\C^2$. We exhibit (Corollary~\ref{Forst}, Figure~\ref{htpyS2}) a Stein domain homotopy equivalent to a 2-sphere, holomorphically embedded in $\C^2$, contradicting a conjecture of Forstneri{\v c} \cite{Fo1}. (Noncompact counterexamples homeomorphic to $S^2\times\R^2$ were exhibited in \cite{JSG}, but these had infinite topology, so were not bounded by smooth 3-manifolds.) We then investigate when a handlebody $H$ consisting of a single 0- and 2-handle embeds as a Stein handlebody in a (minimal) rational ruled surface $S$. An embedding homologous to some section exists if and only if the intersection forms of $H$ and $S$ have the same parity and the attaching circle $K$ of the 2-handle is isotopic to a Legendrian knot with the correct values of $tb$ and $r$ (Corollary~\ref{Hirz}). The choice of complex structure on $S$ is irrelevant. We immediately see (Corollary~\ref{cave}) that if the mirror of $K$ is isotopic to a suitable Legendrian knot, then $H$ admits a complex structure with {\em pseudoconcave} boundary --- so its boundary inherits a negative contact structure (relative to the boundary orientation of the complex surface). We call such a compact, complex surface a {\em pseudoconcave (holomorphic) filling} of its boundary. Every contact 3-manifold has a concave {\em symplectic} filling \cite{Ecap}, \cite{Et}, and if it bounds a Stein domain it has a pseudoconcave holomorphic filling obtained by projectivizing and desingularizing the given Stein surface in $\C^N$. However, the present paper gives pseudoconcave fillings with more directly controlled topology. Corollary~\ref{cave}, for example, shows that if the mirror of $K\subset S^3$ is isotopic to a Legendrian knot with $tb\ge -1$, then any sufficiently large integer surgery on $K$ yields a 3-manifold with a pseudoconcave filling homotopy equivalent to a 2-sphere. Pseudoconcave fillings and embeddings are a recurring theme of this paper, as a simple testing ground for the new technology. We find pseudoconcave contractible manifolds (Section~\ref{Cork}), and pseudoconcave complex structures on manifolds of the form $I\times M^3$ (so both boundary components are pseudoconcave in the boundary orientation determined by the complex structure). Specifically, we find examples of the latter where $M$ is a circle bundle over a surface (following Corollary~\ref{cave}) and where $M$ is a Brieskorn homology sphere (Theorem~\ref{IxSig}). For comparison, $I\times M$ can never admit a pseudoconvex complex structure since it is not homotopy equivalent to a 2-complex, although it is symplectic with boundary weakly convex (or concave) whenever $M\ne S^1\times S^2$ supports a taut $C^2$-foliation \cite{ET}.  Also note that pseudoconcave fillings cannot exist inside a Stein surface (e.g.\ Proposition~\ref{tight}(c)).

In Section~\ref{Brieskorn}, we focus on Brieskorn homology 3-spheres. We show (Theorem~\ref{brieskorn}) that with two exceptions, every Brieskorn sphere $\Sigma(p,q,pq\pm 1)$ (suitably oriented) has a pseudoconvex embedding in every nonspin rational ruled surface, splitting the homology with $b_2=1$ on each side. It follows (Corollary~\ref{briesfill}) that each of these 3-manifolds (suitably oriented) has a pseudoconcave filling homotopy equivalent to $S^2$. The resulting contact structures are usually not homotopic (as plane fields) to the ones arising from the description of Brieskorn spheres as links of algebraic singularities (Proposition~\ref{theta} and following), so it seems unlikely that these embeddings can be constructed algebrogeometrically. We then consider (Theorem~\ref{IxSig} and Corollary~\ref{milnordef}) the family $\Sigma(p,q,npq-1)$. Necessarily excluding $\Sigma(2,3,5)$, we construct pseudoconcave fillings for both orientations, while showing that the 4-manifold $I\times\Sigma(p,q,npq-1)$ admits a complex structure with pseudoconcave boundary. The bottom boundary, with its orientation as the link of a singularity (opposite the boundary orientation from the product), then inherits a positive contact structure, which is the one coming from the singularity. We also see that the corresponding Milnor fiber can be deformed through complex structures, rel an arbitrarily large compact subset, to be pseudoconcave at infinity, rather than pseudoconvex. These results are obtained by compactifying the Milnor fiber in a way generalizing the description of an elliptic surface as a Milnor fiber union a nucleus, then applying Theorem~\ref{main-} to the generalized nucleus. As a byproduct, we obtain a simple polynomial expression for the signature of the $(p,q,npq-1)$ Milnor fiber, namely $-n(p^2-1)(q^2-1)/3$ (Corollary~\ref{sigma}). (More general formulas for signatures of Milnor fibers can be found in the literature, e.g.\ \cite{B}, \cite{Ne}, but these typically involve more complicated expressions such as Dedekind sums or enumeration of lattice points.)

Section~\ref{Cork} shows how various results in the literature can be sharpened in the presence of an ambient complex structure. Akbulut and Mayveyev \cite{AM} showed that a closed, orientable 4-manifold $X$ can always be split into two oppositely oriented Stein domains, glued along their common boundary. The algebraic topology is well controlled, so in the simply connected case, one domain can be taken to be contractible. We show that if $X$ is a complex surface, then the positively oriented piece can be assumed to be holomorphically embedded (Theorem~\ref{AMpsc}). Akbulut and Matveyev showed that their splitting applies to {\em corks}: compact, contractible submanifolds with boundary that can be cut out and reglued to change the smooth structure. (See Section~\ref{Cork} for further discussion.) When one or both of the resulting homeomorphic pair of closed manifolds are complex surfaces, we show (Theorem~\ref{cork}) that the corks can be taken either pseudoconvex or pseudoconcave. In either case, when both manifolds are complex, we can assume the pseudoconvex sides of the splittings are given by the same Legendrian diagram, up to a small correction if their Chern classes differ. This gives an application of controlling the handlebody structure as allowed by Theorem~\ref{main-}. For a second example of sharpening the literature, we consider the paper \cite{AY}, in which Akbulut and Yasui study several applications of cork twists. Of the most interest for the present paper, they showed (see Theorem~\ref{AY}) that under broad hypotheses, a compact 4-manifold (with boundary) embedded in another can be slightly modified to give arbitrarily many embeddings of a fixed manifold in the other, that are pairwise homeomorphic but not diffeomorphic. We show (Theorem~\ref{AYcave}) that when the ambient manifold is complex, the smoothly knotted submanifolds can all be assumed to be pseudoconcave. This section was inspired by Akbulut reminding the author about \cite{AM} and its potential relevance to the present paper.

After a bit of additional background in Section~\ref{Back}, we deal with the main principle in full generality in Section~\ref{Smooth}. In order to deal with infinite and relative Stein handlebodies, we must expand Stein theory to include noncompact complex manifolds with boundary. We call a complex manifold with boundary a {\em Stein shard} if it is cut out of a Stein manifold by a pseudoconvex hypersurface (not necessarily compact) that becomes its boundary (Definition~\ref{shard}). Equivalently, a Stein shard is a complex manifold with pseudoconvex boundary, admitting an exhausting plurisubharmonic function and holomorphically embedding in some open complex manifold of the same dimension (Corollary~\ref{existspsh}). A Stein manifold is obviously the same as a Stein shard without boundary, and in complex dimension $>1$ a Stein domain is the same as a Stein shard with compact boundary that intersects each component nontrivially (Proposition~\ref{tight}(a)). As with Stein domains, the interior of a Stein shard is a Stein manifold (Proposition~\ref{int}), and in complex dimension 2, the boundary is a tight contact manifold (Proposition~\ref{tight}(b)). We define relative Stein handlebodies built on any Stein shard $W$ in such a way that every subhandlebody is again a Stein shard (Definition~\ref{SHB}). These are sufficiently general that every Stein manifold, up to suitable deformation, is the interior of a possibly infinite Stein handlebody with $W=\emptyset$ (Corollary~\ref{all} in complex dimension 2, Corollary~5.2 of \cite{yfest} otherwise, cf.\ Proposition~\ref{SteinHB}). We state and prove the main principle in this generality (Theorem~\ref{main}), which will be needed in \cite{steintop}. Theorems~\ref{smooth} and \ref{main-} follow.

Our notion of a Stein handlebody also seems useful for finite handlebodies with $W=\emptyset$. It is sometimes natural to consider those Stein domains (or manifolds) that arise from Eliashberg's construction. Unfortunately, this notion is rather vague, since the biholomorphism type of a Stein domain changes under small deformations. If we have actually built the handlebody by Eliashberg's method, then the condition is obviously satisfied, but if it arises in a different way, even by a minor variation of the method, the situation is unclear. Instead, we have defined Stein handlebodies using the main consequences of the method, to obtain a more precise and a priori larger collection of Stein domains that in complex dimension~2 are associated to Legendrian link diagrams in the manner arising from Eliashberg's construction.

This paper involves the interplay between manifolds with and without boundary. Our convention is that manifolds are assumed not to have boundary unless otherwise specified. Stein domains and shards can have boundary by definition, as do the closed interval $I=[0,1]$ and disks of any dimension. Thus, handles are compact, and handlebodies have boundary coming from the co-attaching regions. Complex manifolds with boundary involve a subtlety: Their usual definition allows coordinate charts such as $z+\exp({-z^{-\frac12}})$ near 0 on the closed right half-space of $\C$, which is a diffeomorphism onto its image, that is holomorphic (satisfying the Cauchy-Riemann equations) everywhere, but is analytic at $z$ (represented by a power series) only when $z\ne 0$. In higher dimensions, a complex manifold with boundary charts of this sort need not admit a holomorphic embedding into any open complex manifold of the same dimension. However, the complex manifolds with boundary in this paper are always given to be cut out of open complex manifolds by smooth hypersurfaces. Similarly, the contact structures constructed on 3-manifolds in this paper are always tight --- in fact, Stein fillable for some orientation on the 3-manifold (cf.\ Section~\ref{Back}), and the pseudoconcave fillings all embed holomorphically in closed complex manifolds so that the closed complement is a Stein domain.

The author wishes to thank Yasha Eliashberg for many helpful discussions.


\section{Embedding Stein handlebodies} \label{basics}

We begin with a quick sketch of Eliashberg's method for constructing Stein manifolds \cite{E}, \cite{CE}. Suppose $M$ is an oriented real hypersurface (real codimension 1) in a complex $n$-manifold $X$ with  (almost) complex structure $J\co TX\to TX$. Then $M$ inherits a hyperplane field $\xi=TM\cap JTM$ of maximal complex subspaces of the tangent spaces. The condition that $M$ is pseudoconvex (or equivalently, $-M$ is pseudoconcave) implies, and when $n=2$ is equivalent to, the condition that $\xi$ is a positive contact structure on $M$ (negative contact structure on $-M$). This means that $\xi$ is the kernel of some 1-form $\alpha\in\Omega^1(M)$ for which $\alpha\wedge d\alpha\wedge\dots\wedge d\alpha$ is a positive volume form on $M$ (negative volume form on $-M$). Now suppose that $M$ is the boundary (necessarily pseudoconvex) of a Stein domain $W\subset X$, and that $D$ is a real-analytically embedded $k$-disk in $X$ with $D\cap W=\partial D\subset M$. Suppose that $D$ is totally real (i.e.\ the tangent bundle $TD$ contains no complex lines) and intersects $M$ $J$-orthogonally, i.e., transversely with $TD\subset JTM$ along $\partial D$. Eliashberg's main lemma asserts that $D$ is the core of a $k$-handle $h\subset X$ for which $W\cup h$ is again a Stein domain. Note that $J$-orthogonality implies that $T\partial D\subset\xi$. A submanifold of $M$ whose tangent spaces all lie in $\xi$ is called {\em isotropic}, or {\em Legendrian} if $k=n$. If $W$ is instead an abstractly given Stein domain, one can attach a standard holomorphic $k$-handle $D^k\times iD^k_\epsilon\times D^{2n-2k}_\epsilon\subset \R^k\times i\R^k\times \C^{n-k}\subset\C^n$ along a preassigned isotropic $(k-1)$-sphere to obtain a complex manifold with boundary, in which Eliashberg's lemma applies. Replacing the standard handle by the thinner one inside it given by the lemma, we again obtain a Stein domain $W\cup h$. Thus, in either the ambient or abstract setting, we have a method for inductively constructing Stein domains as handlebodies. In either setting, given an arbitrary almost-complex handlebody with all indices $\le n$, where the almost-complex structure in the ambient case is inherited from $X$, there is ultimately only one obstruction to making it into a Stein domain: When $n=k=2$, a smoothly embedded disk $D\subset X-\int W$ cannot always be isotoped to a totally real disk $J$-orthogonal to $\partial W$. This explains the extra complications of the theory when $n=2$.

When $n=2$, we start with a 4-dimensional, compact, almost-complex handlebody $(H,J)$ whose handles have indices $\le 2$, and wish to inductively transform it into a Stein domain. Over the union $H_1$ of 0- and 1-handles, the general theory applies with no difficulty, and we canonically obtain a contact structure on $\partial H_1$, which is a connected sum of copies of $S^1\times S^2$. The attaching circles of the 2-handles form a link in this contact 3-manifold. We can assume this is Legendrian, since every link in a contact 3-manifold is isotopic to a Legendrian link. There is a standard way to represent Legendrian links in  $\partial H_1$ by diagrams \cite{Ann}, \cite{GS}, e.g.\ Figure~\ref{htpyS2} below. Each oriented Legendrian circle $K$ then has two integer-valued invariants, the {\em Thurston-Bennequin invariant} $tb(K)$ and {\em rotation number} $r(K)$. These can be read off of the diagram by counting cusps and crossings. (The invariant $tb(K)$ is the signed number of self-crossings minus the number of left cusps, and the rotation number is half the signed number of cusps, where the sign of a cusp depends on the orientation of the knot, positive when the cusp is traversed downward.) The invariant $tb(K)$ can be decreased arbitrarily by adding zig-zags to the diagram, each of which can either increase or decrease $r(K)$ by 1, but $tb(K)$ cannot always be increased by isotopy. The obstruction to extending a Stein structure over a 2-handle with attaching circle $K$ is that for the core to be totally real, the 2-handle framing coefficient must equal $tb(K)-1$ and the Chern number of $J$ on the (oriented) core, relative to the standard complex framing on $H_1$ (determined by the diagram), must be $r(K)$. It follows that any Legendrian link diagram represents a handlebody realizable as a Stein domain, where the framing on each 2-handle is given by $tb(K)-1$, and the Chern class is given by a cocycle whose value on each oriented 2-handle is the rotation number of its attaching circle. We use the term {\em Stein handlebody} to include any Stein domain exhibited as a handlebody built by Eliashberg's method. We define the term carefully in Section~\ref{Smooth}, in all dimensions and allowing infinite and relative handlebodies built on a Stein shard, but for now we focus on the compact, absolute, 4-dimensional case. For every Stein domain of real dimension 4, every generic plurisubharmonic function that is constant on the boundary determines a Legendrian link diagram (up to Legendrian isotopy), and the original complex structure deforms to one built as a Stein handlebody from that diagram.

We can now state our main principle in the language of Stein handlebodies. The following is a simplified version of Theorem~\ref{main} (cf.\ also \cite{CE} Theorem~13.8).

\begin{thm}\label{main-}
Let $(H,J)$ be a compact, 4-dimensional Stein handlebody with a smooth embedding into a complex surface $X$. Suppose that the complex structure on $H$ induced by the embedding is homotopic (through almost-complex structures) to  $J$. Then after a smooth ambient isotopy of the embedding, the induced complex structure $J'$ on $H$ makes it a holomorphically embedded Stein handlebody with the same Legendrian diagram as $(H,J)$. The two Stein structures determine the same contact structure on $\partial H$.
\end{thm}

\noindent Recall that an {\em ambient} isotopy is obtained by composing the embedding with an isotopy of $\id_X$ through diffeomorphisms of $X$. In particular, the diffeomorphism type of the complement of $H$ is preserved. (The Isotopy Extension Theorem asserts that every smooth isotopy of a compact submanifold extends to such an isotopy with compact support.) Theorem~\ref{main-} says that to realize a compact 4-manifold (with boundary) as a holomorphically embedded Stein domain in a complex surface, it is enough to realize it abstractly as a Stein domain and then find a smooth embedding of it that respects the homotopy class of the complex structure. In general, if $H^2(H;\zed)$ has no 2-torsion, this homotopy class is preserved if and only if the orientation and the Chern class $c_1$ are preserved. This is because the space of complex vector space structures on $\R^4$ determining its usual orientation has the homotopy type of $S^2$, so the first uniqueness obstruction is in  $H^2(H;\zed)$ --- in fact, it is half the difference of the corresponding Chern classes. Since a Stein surface has the homotopy type of a 2-complex, this is the only obstruction (and a homotopy class of almost-complex structures is essentially the same as a spin$^\C$ structure). Most of our examples have no torsion, but Section~\ref{Cork} shows how to apply the theorem when 2-torsion may be present. For another such application, we have:

\begin{cor} If a Stein surface or Stein domain embeds symplectically  into a K\"ahler surface, then the embedding is isotopic to an  embedding (typically not symplectic) whose image is Stein.
\end{cor}

\begin{proof} A symplectic embedding (of the K\"ahler structure induced by some plurisubharmonic function) preserves the homotopy class of the almost-complex structure. Apply Theorem~\ref{smooth} or \ref{main-}, respectively, to the two cases.
\end{proof}

Note that the converse is false. If $C\subset X$ is a complex curve in a K\"ahler surface and $C\cdot C\le -\chi(C)$, then by Theorem~\ref{main-}, a tubular neighborhood of $C$ is isotopic to a Stein domain $U\subset X$ (cf.\ Corollary~\ref{Hirz} and subsequent discussion). However, the K\"ahler form $\omega$ restricted to $U$ is not induced by the Stein domain structure, since it is not exact. (Its value on the generator of $H_2(U)$ must be $\langle\omega,C\rangle>0$.)


\section{Initial applications} \label{Apps}

To warm up, we consider Stein handlebodies with the homotopy type of a point or a 2-sphere. The first case, contractible Stein domains, follows by generalizing our discussion of exotic $\R^4$'s in the Introduction (Corollary~\ref{R4}), and leads to a discussion of pseudoconvex embeddings of homology 3-spheres.

\begin{cor} \label{contractible}
If $V$ is a Stein surface or Stein domain, and $H^2(V;\Z)=0$, then every orientation-preserving embedding of $V$ into a complex surface is isotopic to one with Stein image.
\end{cor}

\begin{proof}
The cohomological condition implies $V$ has a unique homotopy class of almost-complex structures respecting the given orientation. Apply Theorem~\ref{smooth} or \ref{main-} as before.
\end{proof}

\begin{example} \label{hsphere}
It is now easy to give many examples of homology 3-spheres that embed in $\complex^2$ as pseudoconvex boundaries of contractible Stein domains. Let $L \subset M$ be a link in a connected sum of copies of $S^1 \times S^2$. Suppose that $L$ is homotopic (as a map of a disjoint union of circles) to the obvious basis for $\pi_1 (M)$. (More generally, any Andrews-Curtis trivial presentation of the trivial group will work.) After an isotopy, we can assume that $L$ is Legendrian in the standard contact structure on $M$ --- In fact, we typically obtain an infinite family of different Legendrian links this way, by adding zig-zags to vary $tb$. Any such Legendrian link yields a contractible Stein handlebody $H$ that smoothly embeds in $\complex^2$. In fact, $I \times H$ is a 5-ball since its 2-handles become unlinked in this dimension, so its boundary exhibits $H$ embedded  in $\partial B^5=S^4 = \complex^2 \cup \{ \infty \}$. Since the boundary of a compact, contractible manifold is always a homology sphere, Corollary~\ref{contractible} now gives the required embedding. For example, consider a {\em Mazur curve}, a knot in $S^1 \times S^2$ generating its 1-homology. Most Mazur curves have hyperbolic complements. For such a knot, the different choices of Legendrian representative result in infinitely many diffeomorphism types of hyperbolic homology spheres, distinguished by their volumes, with pseudoconvex embeddings in $\complex^2$.
\end{example}

The above embedding problem becomes much harder if we restrict attention to preassigned families of homology spheres. We consider Brieskorn spheres in Section~\ref{Brieskorn}. Many of these do not even embed smoothly in $\C^2$. For example, many have nontrivial Rohlin invariant, so cannot even bound a spin manifold with signature 0. However, Casson and Harer \cite{CH} produced several infinite families of Brieskorn spheres that do bound contractible manifolds, each made from a knot $K$ in $S^1 \times S^2$, and so these do embed in $\C^2$. (Specifically, \cite{CH} gives $\Sigma(p,np+\epsilon,np+2\epsilon)$ for $p$ odd and $\epsilon=\pm1$, and $\Sigma(p,np-1,np+1)$ for $p$ even and $n$ odd, plus the sporadic example $\Sigma(2,3,13)$.) The author has been unable to put Stein structures on any of these contractible manifolds with either orientation, since Legendrian representatives of $K$ seem to always have $tb(K)-1$ strictly smaller than the required framing of the 2-handle. This motivates:

\begin{conj} No Brieskorn sphere with either orientation has a pseudoconvex embedding in $\C^2$.
\end{conj}

\noindent The Casson-Harer examples all bound Stein domains with each orientation, by \cite{Ann}~Theorem~5.4(c) (all $r'_i<-2$ or $r'_1=-2$, $r'_2,r'_3<-3$), but there are further constraints here. Such an embedded Brieskorn sphere would necessarily bound an acyclic Stein domain (note Proposition~\ref{tight}(c)) so, for example, the homotopy class of the induced contact structure (as a plane field) would be uniquely determined ($\theta(\xi)=-2$, cf.\ proof of Proposition~\ref{theta}).

We now shift attention from the homotopy type of a point to that of a 2-sphere. Forstneri{\v c} \cite{Fo1} conjectured that no domain of holomorphy in $\C^2$ could have the homotopy type of $S^2$. Counterexamples homeomorphic to $S^2\times \R^2$ were constructed in \cite{JSG}, realizing uncountably many diffeomorphism types \cite{steintop}. However, these necessarily had infinite topology, and could not even be realized up to diffeomorphism as interiors of compact manifolds. Now we require the region to be cut out by a compact, pseudoconvex 3-manifold, by restricting attention to embedded Stein domains. One cannot construct such an example using just a 0-handle and 2-handle, for if we could, the attaching curve would be a slice knot. (To see this, compactify $\C^2$ to $S^4$, then remove the interior of the 0-handle, leaving behind a 4-ball in which the knot explicitly bounds an embedded disk, namely the core of the 2-handle.) A slice disk in the 0-handle would then fit together with the core of the 2-handle to realize a homologically nontrivial sphere of square 0 in the Stein surface, contradicting gauge theory \cite{LM}. On the other hand, if we allow one 1-handle, the situation changes:

\begin{cor} \label{Forst}
There is a Stein domain holomorphically embedded in $\C^2$  that is homotopy equivalent to the 2-sphere.
\end{cor}

\begin{proof} Figure~\ref{htpyS2}
\begin{figure}
\centerline{\epsfbox{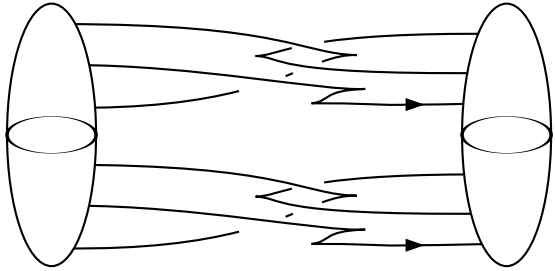}}
\caption{A Stein homotopy $S^2$ in $\C^2$}\label{htpyS2}
\end{figure}
 shows a Stein handlebody $H$ made with one 0-handle (whose boundary is the background $S^3=\R^3\cup\{\infty\}$), one 1-handle (attached at the two ellipsoids), and two 2-handles attached along a pair of oriented Mazur curves in the resulting $S^1\times S^2$ boundary. The curves are both Legendrian, with $tb=r=1$. Since the curves are both homotopic to $S^1\times \{p\}$ in $S^1\times S^2$, $H$ has the homotopy type of $S^2$. Its Chern class is 0, given by the difference of the rotation numbers since the difference of the 2-handles generates $H_2(H)$. If we ignore Stein structures, we can add a 2-handle to cancel the 1-handle, leaving a 0-framed unlink in $S^3$. Adding two 3-handles gives $B^4\subset \C^2$, so we have smoothly embedded $H$ in $\C^2$.  Since there is no 2-torsion and the embedding preserves the Chern class, it preserves the almost-complex structures up to homotopy. Theorem~\ref{main-} completes the proof.
\end{proof}

In contrast to $\C^2$, other complex surfaces do admit holomorphically embedded Stein handlebodies consisting of a 0-handle and 2-handle. We consider the (minimal) rational ruled surface $S_m$ with holomorphic sections of square $\pm m$ under the intersection pairing. Recall that such surfaces are holomorphically distinct for different values of $m=0,1,2,\dots$, but smoothly (or symplectically) the underlying $S^2$-bundles over $S^2$ are determined by the mod 2 residue of $m$. In particular, $S_m$ admits smooth (symplectic) sections of all squares congruent to $m$ mod 2.

\begin{cor}\label{Hirz}
Suppose that $H$ is obtained by attaching a 2-handle to $B^4$ along a knot $K$ with framing $n$. Fix a generator of $H_2(H;\Z)\cong\Z$ by orienting $K$. Then $H$ can be realized in $S_m$ as a Stein handlebody representing the homology class of some smooth section, if and only if $m\equiv n$ mod 2 and $K$ has a Legendrian representative  with $tb=n+1$ and $r=n+2$.
\end{cor}

\noindent Note that changing the orientation of $K$ reverses the sign of $r$ and changes the embedding problem, since the sections of $S_m$ are canonically oriented (intersecting the holomorphic fibers positively).

\begin{proof} To construct such an embedding, we simply double $H$, i.e., look at $\partial (I\times H)$. The resulting 4-manifold is an $S^2$-bundle over $S^2$, trivial if and only if $n$ is even, since the attaching circle becomes unknotted when the dimension increases. To see this with Kirby calculus, we simply add a 0-framed meridian to $K$, which allows $K$ to be unknotted by handle slides (e.g.\ \cite{GS}). Since the meridian represents a fiber of the bundle throughout the computation, and has linking number 1 with $K$, the homology class carried by $H$ has intersection number 1 with the fiber class and so is represented by a smooth section. For any $m\equiv n$ mod 2, we can identify the bundle with $S_m$ (preserving the ambient and fiber orientations). The resulting complex structure on $H$ is determined up to homotopy by its Chern number. This must be $n+2$ by the adjunction formula, since the homology class has square $n$ and is represented by a symplectic sphere in $S_m$. We realize the same Chern number by a Stein structure on $H$ using the given Legendrian representative of $K$. (The condition on $tb$ guarantees that we get the correct manifold, and the one on $r$ guarantees the correct almost-complex structure.) Theorem~\ref{main-} now gives the required Stein handlebody.

Conversely, if $H$ is given as a Stein handlebody, then $K$ is Legendrian with $tb=n+1$, and the given embedding guarantees the other conditions.
\end{proof}

\begin{remark} The corollary applies without change if $H$ is also allowed to have 1-handles, provided that $K$ is nullhomologous (so that $H$ still determines a homology class). The only complication in the proof is that we must add 2-handles to cancel the 1-handles before doubling. Similarly, if $F$ is a Riemann surface and $K$ is homotopic to the attaching circle in the usual handle decomposition of $F\times D^2$, we obtain the analogous theorem with $S_m$ replaced by a ruled surface over $F$ and $r=n+\chi(F)$.
\end{remark}

Looking at the complement of $H$ in $S_m$ gives our first result on pseudoconcave fillings:

\begin{cor}\label{cave}
For $K$ and $H$ satisfying the hypotheses of Corollary~\ref{Hirz} (for some orientation of $K$), the handlebody $-H$ on the mirror of $K$ with framing $-n$ admits a complex structure with pseudoconcave boundary. In particular, for any knot $K'\subset S^3$ whose mirror $-K'$ has a Legendrian representative  with $\pm r=tb+1$, surgery on $K'$ with coefficient $1-tb(-K')$ yields a 3-manifold with a pseudoconcave filling homotopy equivalent to $S^2$. \qed
\end{cor}

\noindent To prove the last sentence, apply the first with $K=-K'$. Note that the last sentence applies to any knot $K'$ for which $-K'$ has a Legendrian representative with $tb\ge-1$, and then for any sufficiently large integral surgery on $K'$. (Since $tb$ and $r$ have opposite parity for any Legendrian knot in $S^3$, we can choose $\pm r\le0$ and then reduce $tb$ as necessary fixing $r$. Then we can also realize any smaller $tb$, maintaining $\pm r=tb+1$.) A similar corollary applies in the cases of the previous remark, although in the first case the pseudoconcave filling replacing $-H$ is more complicated, and in the second case $F$ replaces $S^2$. The latter case also gives our first approach to constructing pseudoconcave complex structures on manifolds $I\times M^3$, succeeding whenever $M$ is a circle bundle over $F$ with Euler number $e$ for which $|e|\le -\chi(F)$: Simply take the complement of a pair of disjoint Stein handlebodies isotopic to tubular neighborhoods of sections with opposite square in a ruled surface over $F$. The required Stein handlebody structures on disk bundles over $F$ were constructed in \cite{Ann}; see also \cite{GS}. It is not clear whether this approach generalizes, since any modification of the attaching circles, such as summing with a knot, would have to respect the invariants with both orientations. In fact, when $F=S^2$ the approach fails, since it is impossible to find knots $K$ to which the above hypotheses apply with both orientations: The required framings would be $\pm n$ for some $n$, implying $tb(K)+tb(-K)=2$ for the given Legendrian representatives. However, any knot in $S^3$ satisfies $tb(K)+tb(-K)\le -2$ for all Legendrian representatives. (The connected sum $K^*=K\#-K$ is slice, with $tb(K^*)=tb(K)+tb(-K)+1$ if the sum is suitably performed in the Legendrian setting, so the formula follows from the slice-genus inequality $tb(K^*)+|r(K^*)|\le2g_s(K^*)-1$ of \cite{R}; see Corollary~11.4.9 of \cite{GS} for a proof of the latter inequality via Stein theory.) When $M=S^1\times S^2$, we obtain a pseudoconcave structure by removing a Stein neighborhood of $S^1\times\{ \text{2  points}\}$ from $S^1\times S^3=(\C^2-\{0\})/(\times 2)$. We solve the problem for $M$ in a family of Brieskorn spheres in the next section (Theorem~\ref{IxSig}).


\section{Brieskorn spheres} \label{Brieskorn}

We now apply Theorem~\ref{main-} to study pseudoconvex and pseudoconcave embeddings of Brieskorn spheres. Recall that for pairwise relatively prime integers $p_1,p_2,p_3\ge 2$, the {\em Brieskorn sphere\/} $\Sigma(p_1,p_2,p_3)$ is the Seifert fibered homology 3-sphere with multiplicities $p_1,p_2,p_3$. It is canonically oriented as the link of the singularity of the variety $z_1^{p_1}+z_2^{p_2}+z_3^{p_3}=0$ in $\C^3$. That is, we intersect the variety with the unit 6-ball to obtain an oriented singular 4-manifold whose boundary in $S^5$ is $\Sigma(p_1,p_2,p_3)$ with its canonical orientation. Replacing 0 by a small nonzero value $\epsilon$ in the above equation realizes $\Sigma(p_1,p_2,p_3)$ as the boundary of a Stein domain called the {\em Milnor fiber} $\Phi(p_1,p_2,p_3)$. (We abusively use the same notation for the manifold with boundary and the entire variety in $\C^3$.) One also obtains $\Sigma(p_1,p_2,p_3)$ by surgery on three fibers of $S^2\times S^1$, with surgery coefficients (relative to the product framing) $p_i/q_i$ for any choice of $q_1,q_2,q_3$ such that $q_1p_2p_3+p_1q_2p_3+p_1p_2q_3=+1$. (When the sum is $-1$, we obtain the other orientation.) It is well-known (at least up to orientation) how to realize some Brieskorn spheres by surgery on torus knots: For  a positive integer $n$, $\pm 1/n$-surgery on the right-handed $(p,q)$-torus knot $T_{p,q}$ yields $\mp\Sigma(p,q,npq\mp 1)$. This can be verified by identifying the complement of two fibers in $S^2\times S^1$ with the Hopf link complement in $S^3$ so that a third fiber maps to the torus knot. To work with handlebodies with a single 2-handle, we first restrict to the integral case $n=1$.

\begin{thm}\label{brieskorn} Let $p,q$ be relatively prime integers with $2\le p<q$ and let $\epsilon=\pm1$. When $(p,q,pq+\epsilon)\ne (2,3,5)$ or $(2,5,9)$,  the Brieskorn sphere $\Sigma(p,q,pq+\epsilon)$ has a  pseudoconvex (for $\epsilon=+1$) or pseudoconcave (for $\epsilon=-1$) embedding into each rational ruled surface $S_m$ with $m$ odd, splitting it into two simply connected regions with intersection pairings $\langle\pm 1\rangle$, with the canonical orientation of the Brieskorn sphere agreeing with the boundary orientation of the negative side.
\end{thm}

Beware that there are four orientations here to compare on the Brieskorn sphere: the canonical orientation from the singularity, the contact orientation from the given embedding, and the boundary orientations of the two pieces cut out of $S_m$. The distinction between pseudoconvexity and pseudoconcavity comes from comparing the contact orientation to some preassigned orientation on the hypersurface. Elsewhere in this paper, the hypersurface arises as the boundary of a complex manifold, giving it the required preassigned orientation. Now, there are two conflicting boundary orientations, so we use the canonical orientation of the Brieskorn sphere --- which happens to be the boundary orientation of the side with negative definite intersection form.

\begin{proof} Since $p$ and $q$ are relatively prime and $\ge2$, we have $(p-1)(q-1)=2l$ for some positive integer $l$. The obvious diagram of the knot $T_{p,q}$ can be made Legendrian with $p$ left cusps and $(p-1)q$ crossings, all positive. Thus, it has $tb=(p-1)q-p=2l-1$. Since the numbers of upward and downward cusps are equal, $r=0$. Adding $l-1$ upward zig-zags and $l$ downward zig-zags, we obtain $tb=0$ and $r=1$. Thus, Corollary~\ref{Hirz} applies to the handlebody $H$ determined by $T_{p,q}$ with framing $n=-1$. We obtain $H\subset S_m$ representing the homology class of the smooth section of square $-1$. Since $S_m$ was constructed in that proof by doubling, the complement of $\int H$ is $H$ with reversed orientation, so both regions are simply connected. The boundary of $H$ is $-1$-surgery on $T_{p,q}$, which is $\Sigma(p,q,pq+1)$. This completes the $\epsilon=+1$ case. For the $\epsilon=-1$ case, note that unless $(p,q)=(2,3)$ or $(2,5)$, we have $l\ge3$. Adding $l-3$ upward zig-zags and $l$ downward zig-zags to the original Legendrian diagram gives $tb=2$ and $r=3$. Now the previous argument applies with $n=+1$. We obtain $-\Sigma(p,q,pq-1)$ realized as the pseudoconvex boundary of the positive region, and the theorem follows.
\end{proof}

\begin{cor} \label{briesfill}
For $p,q,\epsilon$ satisfying the previous hypotheses, the Brieskorn sphere $-\epsilon\Sigma(p,q,pq+\epsilon)$ has a pseudoconcave  holomorphic filling diffeomorphic to the handlebody determined by the left-handed torus knot $-T_{p,q}$ with framing $\epsilon$, and with Chern number $2+\epsilon$.
\qed
\end{cor}

We consider a few contrasting examples. The 3-sphere embeds as in the theorem precisely when $m=1$, and only in the pseudoconvex case. The example is obtained from $S^3=\partial B^4\subset\CP^2$ by blowing up an interior point of $B^4$. Unlike for the Brieskorn spheres, the pseudoconvex domain is not Stein, but a blowup of the Stein domain $B^4$. Other embeddings of $S^3$ as in the theorem are ruled out since the only holomorphic (pseudoconvex) fillings of $S^3$ are blowups of $B^4$ \cite{Durham}. The Brieskorn sphere $\Sigma(2,3,5)$ has no pseudoconcave embedding as in the theorem or tight filling as in the corollary, since it admits no negative tight contact structure \cite{EtH}. The remaining excluded example $\Sigma(2,5,9)$ cannot be realized by the above method using $T_{2,5}$, because of the genus inequality $2+3=tb+|r|\le 2g-1$, since $T_{2,5}$ has genus $g=2$. However, this leaves open the possibility of realization by other knots or handlebodies with more handles. The above method at least shows that $-\Sigma(2,5,9)$ does bound a Stein handlebody with the right homology, made from  $T_{2,5}$ with framing +1, but with the wrong Chern class (given by $r=1$) for our application. The theorem also leads us to consider embeddings with the opposite orientation:

\begin{ques} Is there a pseudoconvex or pseudoconcave embedding of a Brieskorn sphere in a rational ruled surface, as the oriented boundary of a region with intersection pairing $\langle+1\rangle$?
\end{ques}

\noindent A natural approach would be to apply the previous method to left-handed torus knots. However, a Legendrian representative of such a knot must have negative $tb$, so the resulting handlebody with framing $\pm 1$ cannot be a Stein handlebody. Note that $\Sigma(2,3,5)$ admits no embedding as in the question, even at the level of smooth manifolds, for otherwise we could cut out the positive region and glue in the Milnor fiber to obtain a 4-manifold with the negative definite, nonstandard form $\langle-1\rangle\oplus E_8$, violating Donaldson's Theorem \cite{D}. We can also ask whether any Brieskorn sphere has two isotopic embeddings realizing pseudoconvexity in both directions, or more strongly, whether $I\times\Sigma(p,q,r)$ has a pseudoconcave embedding into any $S_m$, splitting the homology nontrivially. We give examples of pseudoconcave embeddings of such products into other complex surfaces in Theorem~\ref{IxSig}. Note that only the mod 2 residue of $m$ affects the answers except when $m=1$ and the pseudoconvex region contains the exceptional sphere. Otherwise, this region will be a Stein domain, and the almost-complex (and symplectic) structure only depends on $m$ mod 2, so Theorem~\ref{main-} proves the observation.

Every Brieskorn sphere inherits a contact structure as the pseudoconvex boundary of its Milnor fiber, so it is natural to compare these structures with the ones inherited via Theorem~\ref{brieskorn}. For $\epsilon=-1$ the structures are clearly different, since they induce opposite orientations on the 3-manifold, but it is less clear how the structures compare when $\epsilon=+1$. We now show that in both cases, the structures frequently cannot even be homotopic as plane fields.

\begin{prop}\label{theta} If $\epsilon=1$ and $p,q\equiv 2$ (mod 3) then the contact structure induced by Theorem~\ref{brieskorn} is not homotopic through plane fields to the one induced by the Milnor fiber. If $\epsilon=-1$, the two structures are homotopic if and only if $(p,q,pq+\epsilon)=(2,7,13)$ or $(3,4,11)$.
\end{prop}

\begin{proof} Since oriented planes through 0 in $\R^3$ are determined by their positive unit normal vectors, oriented plane fields $\xi$ on a homology 3-sphere $\Sigma$ are classified up to homotopy by $\pi_3(S^2)\cong\Z$ (up to translation). The homotopy classes are determined by the invariant $\theta$ of \cite{Ann}: For any compact, almost-complex $(X,J)$ with boundary $(\Sigma,\xi)$, we have $\theta(\Sigma,\xi)=c_1^2(X,J)-2\chi(X)-3\sigma(X)$. The invariant on a given homology sphere ranges over all integers congruent to 2 mod 4. (See the bottom of p.\ 445 in \cite{GS}.) We check that each contact structure induced on a Brieskorn sphere $\Sigma$ by Theorem~\ref{brieskorn} has $\theta=-2$: For $\epsilon=+1$, we exhibited $\Sigma$ as the boundary of a Stein handlebody $H$ with intersection form $\langle -1\rangle$, given by a Legendrian knot with $r=1$. This value of $r$ implies that $c_1(H)$ is a generator of $H^2(H)\cong\Z$, so $\theta=-1-4+3=-2$. For $\epsilon=-1$, we obtain $-\Sigma$ as the boundary of a Stein $H$ with form $\langle +1\rangle$ and $r=3$. Now $\theta(-\Sigma)=3^2-4-3=2$, but reversing orientation on $\Sigma$ reverses the sign of $\theta$.

 The Milnor fiber $\Phi=\Phi(p,q,r)$ is the preimage of a regular value under a complex polynomial, so its normal bundle in $\C^3$ is trivial, as is $T\C^3|\Phi$. Thus, $c_1(\Phi)=0$, and the induced contact structure on $\Sigma(p,q,r)$ has $\theta=-2\chi(\Phi)-3\sigma(\Phi)$. Since $\Phi$ is homotopy equivalent to a wedge of 2-spheres, $\chi(\Phi)=b_2(\Phi)+1$, and we can show $\theta\ne-2$ by checking that $b_2(\Phi)\not\equiv0$ mod 3. But $b_2(\Phi)=(p-1)(q-1)(r-1)$ \cite{B}, so the congruence follows from the hypotheses that $r=pq+\epsilon$ and $p,q\equiv 2$ (mod 3) (or that $\epsilon=-1$ and $p,q\not\equiv 1$ (mod 3)). For $\epsilon=-1$, we explicitly compute $\theta$ for $(p,q,npq-1)$ in Corollary~\ref{sigma} below. Setting $n=1$, we see that $\theta(\partial\Phi)=-2$ precisely when $(p-1)(q-1)=6$, i.e., $(p,q)=(2,7)$ or $(3,4)$.
\end{proof}

In fact, the difference of the induced plane fields can be guaranteed to be arbitrarily large, just by taking $p$ sufficiently large. This is clear from Corollary~\ref{sigma}(a) when $\epsilon=-1$, and can be seen in general from Brieskorn's expression \cite{B} for $\sigma(\Phi)$ as a signed count of lattice points in a cube in $\R^3$, if we use volumes to estimate the count when the lattice is sufficiently fine.

Now we set $\epsilon=-1$ and expand our family of Brieskorn spheres by allowing arbitrary $n$. We construct pseudoconcave (and pseudoconvex) fillings with both orientations, by finding a pseudoconcave embedding of the product with $I$ in a suitable generalization of an elliptic surface. This also yields a peculiar deformation of the corresponding Milnor fibers (Corollary~\ref{milnordef}(b)).

\begin{thm}\label{IxSig} For $n,p,q\in\Z_+$ with $2\le p<q$ and $p,q$ relatively prime, the 4-manifold $I\times \Sigma(p,q,npq-1)$ has a complex structure with pseudoconcave boundary if and only if $(p,q,npq-1)\ne (2,3,5)$. The induced positive contact structure on $\{0\}\times \Sigma(p,q,npq-1)$ (which is negative in the boundary orientation $-\Sigma(p,q,npq-1)$) is the one canonically determined by the Milnor fiber.
\end{thm}

\noindent Note that the complex structure induces a plane field on each $\{t\}\times \Sigma(p,q,npq-1)$, so the boundary contact structures are homotopic through plane fields.

\begin{proof} We construct a closed complex surface containing the given product manifold. For $k\in\Z$, let $\pi\co L_k\to\CP^1$ be the line bundle with Chern number $k$. Let $\tilde\pi\co \pi^*L_{nq}\to L_{np}$ be the pullback of $L_{nq}$ by $\pi$ for $k=np$. The total space $\pi^*L_{nq}$ is covered by two charts $\psi_i$, $i=0,1$, with coordinates $z_i$ (for $\CP^1$), $y_i$ (for the fibers of $\pi$ on $L_{np}$) and $x_i$ (for the fibers of $\tilde\pi$), related by $z_1=z_0^{-1}$, $y_1=z_0^{-np}y_0$ and $x_1=z_0^{-nq}x_0$. Consider the variety in this space given by
$$X=\{x_1^p+y_1^q+z_1^{npq-1}=\epsilon\}=\{x_0^p+y_0^q=(\epsilon z_0^{npq-1}-1)z_0\}\subset\pi^*L_{nq}$$
for a fixed $\epsilon\ne0$. This is a smooth complex surface, since it restricts to the Milnor fiber $\Phi(p,q,npq-1)$ under $\psi_1$ and the graph of $x_0^p+y_0^q=z_0$ (after a holomorphic coordinate change) near $z_0=0$. The restriction $\tilde\pi\co X\to L_{np}$ is a $p$-fold branched covering map, since any fixed $(y_i,z_i)$ corresponds to $p$ values of $x_i$, except on the branch locus $X\cap\{x_i=0\}$ over
$$C=\{y_1^q+z_1^{npq-1}=\epsilon\}=\{y_0^q=(\epsilon z_0^{npq-1}-1)z_0\}\subset L_{np}$$
in the base $L_{np}$. The curve $C$ is also smooth, since it is given by a Milnor fiber in $\C^2$ (essentially a Seifert surface of $T_{q,npq-1}$) in one chart, and a graph (up to coordinate change) near $z_0=0$. Furthermore, the restriction $\pi\co C\to\CP^1$ is a $q$-fold branched covering map, with branch locus given by the $npq$ generic roots of $(\epsilon z_0^{npq-1}-1)z_0$. In particular, $C$ is compact. Thus, we have exhibited $X$ as the $p$-fold cover of $L_{np}$ branched over the compact curve $C$. Over each fiber of $L_{np}$ (i.e.\ $z_i$ constant), the corresponding fiber in $X$ is the $(p,q)$-Milnor fiber, exhibited as a $T_{p,q}$ Seifert surface in the $x_i$-$y_i$ plane, except that at the $npq$ roots of $(\epsilon z_0^{npq-1}-1)z_0$ we see a fiber with a singularity that is a cone on $T_{p,q}$. In either case, the complex curve has only a single puncture at infinity (since $p$ and $q$ are relatively prime so that $T_{p,q}$ has a single component). Thus, a disk subbundle of $L_{np}$ large enough to contain $C$ has boundary a lens space with $\pi_1=\Z/np$, over which $\tilde\pi$ is the unique connected $p$-fold covering. If we compactify $L_{np}$ to the rational ruled surface $S_{np}$ by adding a section $C_\infty$ of square $-np$ at infinity, the branched covering $\tilde\pi$ extends to $\bar\pi\co\bar X\to S_{np}$, a $p$-fold covering branched over $C\cup C_\infty$. The cover $\bar X$ is a compact complex surface containing the original Milnor fiber as the complement of a pair of curves: the lift $\sigma$ of $C_\infty$ and the curve $F$ obtained by compactifying $X\cap\{ z_0=0\}$. The smooth sphere $\sigma$ has square $-n$, and is a section of the extended projection $\pi\circ\bar\pi\co \bar X\to \CP^1$. The curve $F$ is a sphere with one singularity that is a cone on $T_{p,q}$. (In fact, $\pi\circ\bar\pi$ exhibits a singular fibration on $\bar X$ with fiber genus $(p-1)(q-1)/2$ and $npq$ singular fibers diffeomorphic to $F$. For $(p,q)=(2,3)$, we recover the usual description of $\bar X=E(n)$ as an elliptic surface with $6n$ cusp fibers.)

Now let $N\subset \bar X$ be obtained by deleting the intersection of $\bar X$ with a sufficiently large round open ball in the chart given by $\psi_1$. Then $N$ is a regular neighborhood of $\sigma\cup F$, with pseudoconcave boundary $-\Sigma(p,q,npq-1)$, and closed complement the compact Milnor fiber $\Phi(p,q,npq-1)$. The negative contact structure on $\partial N$ equals the positive contact structure on $+\Sigma(p,q,npq-1)$ induced by the Milnor fiber. But $\sigma \cdot \sigma=-n$ and $F\cdot F=0$, since $F$ is homologous to a disjoint regular fiber given by constant $z_0$ in $\bar X$. Thus, $N$ is the handlebody given by $T_{p,q}$ with framing 0 and a $-n$-framed meridian (verifying directly that $\partial N$ is $\frac1n$-surgery on $T_{p,q}$ as required). By the adjunction formula, $\langle c_1(\bar X),\sigma\rangle=2-n$ and $\langle c_1(\bar X),F\rangle=2-(p-1)(q-1)=2-2l$ in the notation of the proof of Theorem~\ref{brieskorn}. As observed there, $T_{p,q}$ has a Legendrian representative with $tb=2l-1$ and $r=0$. Adding $2l-2$ zig-zags gives $tb-1=0$ and $r=2-2l$. If $n\ge 2$, we can add a Legendrian meridian with $tb-1=-n$ and $r=2-n$, exhibiting the smooth manifold $N$ as a Stein domain whose Chern class agrees with that obtained from the embedding in $\bar X$. Now Theorem~\ref{main-} gives a compactly supported ambient isotopy of the complex surface $\int N$ sending a smaller copy of $N$ to a Stein domain $N'\subset N$, and $N-\int N'$ is the required pseudoconcave complex surface diffeomorphic to $I\times \Sigma(p,q,npq-1)$. If $n=1$, we blow down $\sigma$ to obtain $N^*\subset X^*$, with $N^*$ given by $T_{p,q}$ with framing 1. Now we must find a Legendrian representative of $T_{p,q}$ with $tb=2$ and $r=3-2l$. We can get this from our original Legendrian diagram by adding $2l-3$ zig-zags, as long as $2l-3\ge0$, or equivalently, $(p,q)\ne (2,3)$. In that remaining case, we are dealing with $\Sigma(2,3,5)$, so no negative tight contact structure exists.
\end{proof}

The proof actually gives us much more information:

\begin{cor} \label{milnordef}
For $p,q,n$ as in Theorem~\ref{IxSig},
\item[a)] the Brieskorn sphere $-\Sigma(p,q,npq-1)$ has a pseudoconcave filling, by the handlebody on $T_{p,q}$ with framing $+1$ if $n=1$, or on $T_{p,q}$ with framing 0 and a meridian with framing $-n$ in general. The induced contact structure is given by the Milnor fiber. If $(p,q,npq-1)\ne(2,3,5)$, the identity map on the filling is isotopic (in the filling) to an embedding with pseudoconvex boundary.
\item[b)] If $(p,q,npq-1)\ne(2,3,5)$,  the usual (Stein) complex structure on the open Milnor fiber $\Phi(p,q,npq-1)$ (for any fixed $\epsilon\ne0$) can be deformed, through a real 1-parameter family of complex structures agreeing on a preassigned compact subset, to a complex structure on which the level sets of all sufficiently large values of the radial function (restricted from $\C^3$) are pseudoconcave. (This gives a pseudoconcave filling of $\Sigma(p,q,npq-1)$.)
\qed
\end{cor}

\begin{cor}\label{sigma}
For $p,q,n$ as in Theorem~\ref{IxSig}, the Milnor fiber $\Phi=\Phi(p,q,npq-1)$ and its contact boundary satisfy:
\item[a)] $\theta(\partial \Phi)=(p-1)(q-1)(4-n(pq-p-q-1))-2$
\item[b)]  $\sigma(\Phi)=-n(p^2-1)(q^2-1)/3$.
\end{cor}

\begin{proof} In the proof of Theorem~\ref{IxSig}, we evaluated $c_1(\bar X)$ on the surfaces $\sigma, F$ giving a basis for $H_2(N;\Z)$. Restricting to $N$, we see that the Poincar\'e dual of $c_1(N)$ is given by $(2-2l)\sigma+(2+n(1-2l))F$ (since the latter class pairs correctly with $\sigma, F$). Thus, $c_1^2(N)=(2-2l)(4-2nl)$. Since $\partial N= \partial \Phi$ as contact manifolds, but the boundary orientations are opposite, the invariant discussed in the proof of Proposition~\ref{theta} satisfies $\theta(\partial \Phi)=-\theta(\partial N)=-c_1^2(N)+2\chi(N)+3\sigma(N)=-c_1^2(N)+2\cdot 3+3\cdot 0 =4l(2-n(l-1))-2$, which gives (a) after the substitution $2l=(p-1)(q-1)$. For (b), note (as in the proof of Proposition~\ref{theta}) that $c_1(\Phi)=0$ and $\chi(\Phi)=1+2l(npq-2)$. Thus, $3\sigma(\Phi)=c_1^2(\Phi)-2\chi(\Phi)-\theta(\partial\Phi)
=-4nl(pq-l+1)$, giving the required formula.

A slightly more direct approach to (b) is to observe that gluing $N$ to $\Phi$ increases $\chi$ of the latter by 3 while fixing $\sigma$, and the above formula for $c_1^2(N)$ also gives $c_1^2(\bar X)$ since $c_1(\Phi)=0$ and $\partial N$ is a homology sphere, so $\sigma(\Phi)=\sigma(\bar X)$ can be computed from the formula $c_1^2(\bar X)=2\chi(\bar X)+3\sigma(\bar X)$.
\end{proof}

\begin{remark} When $(p,q)=(2,3)$, we recover the elliptic surface $\bar X=E(n)$ decomposed into its {\em nucleus} $N$ and the Milnor fiber $\Phi(2,3,6n-1)$, as originally presented topologically in \cite{Nuc}. For an alternate holomorphic presentation of this case, see \cite{EO}, which also implicitly suggests how to generalize Theorem~\ref{IxSig} and its corollaries to other algebraic singularities such as $(2,3,6n+1)$. ($N$ is replaced by a plumbing with larger $b_2$.)  For $(2,3,6n-1)$ with $n>1$, the proof of Theorem~\ref{IxSig} realizes both pieces of the elliptic surface as Stein domains, separated by an embedding of $I\times\Sigma(2,3,6n-1)$ with pseudoconcave boundary. One can obtain other elliptic surfaces from $E(n)$ by logarithmic transformation on regular fibers, which we can assume lie in $N$. It is natural to ask if these elliptic surfaces have analogous decompositions along  $I\times\Sigma(2,3,6n-1)$. Since logarithmic transformation in $N$ preserves the holomorphic structure of the Milnor fiber on an arbitrarily large compact subset, convexity of that side is preserved. However, Yasui~\cite{Y} has recently shown that after one or two nontrivial logarithmic transformations, $N$  can never be realized as a Stein domain.
\end{remark}


\section{Akbulut corks} \label{Cork}

This section shows how various published theorems can be sharpened to include pseudoconvexity or pseudoconcavity as conclusions. The proofs show how Theorem~\ref{main-} can be applied in situations where there may be 2-torsion complicating the classification of almost-complex structures. A good source of examples is Akbulut and Matveyev \cite{AM}. That paper proved that every closed, oriented 4-manifold $X$ can be decomposed as $\tilde X_1\cup_\partial \tilde X_2$ so that $\tilde X_1$ and $-\tilde X_2$ both abstractly admit the structure of Stein domains, with control over the algebraic topology of the decomposition. They showed that if $X$ is simply connected then $\tilde X_2$ can be taken to be contractible. Moreover, they showed that such a description can always be arranged in the construction of {\em Akbulut corks}. If $Z_1$ and $Z_2$ are closed, simply connected 4-manifolds with the same intersection form, then they are necessarily related by a {\em cork twist} --- removing a compact, contractible 4-manifold (a {\em cork}) from $Z_1$ and regluing it by a diffeomorphism of the boundary to obtain $Z_2$. (The first example of this was discovered by Akbulut \cite{A}; the general case was proved in \cite{CFHS}, \cite{M}.) Akbulut and Matveyev showed that the cork and its complement can always be chosen so that both pieces (suitably oriented) abstractly admit the structure of Stein domains. We now show that the Stein domain structures of \cite{AM} can all be constructed ambiently. That is, whenever the 4-manifold $X$ is a complex surface, we can assume that $\tilde X_1\subset X$ is a holomorphically embedded Stein domain, so $\tilde X_2$ has pseudoconcave boundary. In particular, corks in complex surfaces can always be chosen to have pseudoconcave boundary. As another example, we sharpen a knotting theorem of Akbulut and Yasui \cite{AY} by arranging the knotted objects to be pseudoconcave when the ambient space is complex.

We begin with the main theorem of \cite{AM}.

\begin{thm} \label{AM} \cite{AM}. Let $X=X_1\cup_\partial X_2$ be a closed, oriented 4-manifold, decomposed into a pair of handlebodies without 3- or 4-handles. Then there is another such decomposition  $X=\tilde X_1\cup_\partial \tilde X_2$ such that $\tilde X_1$ and $-\tilde X_2$ are Stein handlebodies, and each $\tilde X_i$ has a homotopy equivalence to $X_i$ preserving the intersection pairing.
\end{thm}

\noindent The homotopy equivalences are not pairwise on  $(\tilde X_i,\partial\tilde X_i)$, since the fundamental group of the boundary changes in general.

\begin{thm} \label{AMpsc} If the above $X$ is complex, then we can assume $\tilde X_1$ is holomorphically embedded in $X$. If $-X$ is also complex, then after an isotopy of its complex structure, we can also assume $-\tilde X_2$ is holomorphically embedded in $-X$.
\end{thm}

\noindent Of course, it is somewhat rare for both $X$ and $-X$ to admit complex structures. However, many examples can be obtained by considering bundles over surfaces, or elliptic fibrations with only smooth (possibly multiple) fibers.

\begin{proof} Akbulut and Matveyev construct the submanifolds $\tilde X_i$ by a sequence of operations adding a 2-handle to some $X_j$ while removing it from its complement in $X$. The net effect is to introduce some algebraically canceling 1-2 handle pairs onto each $X_i$, then divert the attaching circles of some original 2-handles over new 1-handles. The key point is that each 2-handle of each of the new handlebodies $\tilde X_i$ can be chosen flexibly. Each 2-handle is chosen from an infinite collection of possible 2-handles, whose attaching circles are all homotopic in the boundary of the union $H$ of 0- and 1-handles of $\tilde X_i$, and whose homology classes in $H_2(X,H)$ are equal. Choosing the 2-handles suitably allows both $\tilde X_1$ and $-\tilde X_2$ to be represented by Stein handle diagrams as required. Closer inspection shows that one can independently choose each Legendrian attaching circle of each Stein surface from an infinite family with fixed rotation number but arbitrarily many extra zig-zags in both directions.

To apply Theorem~\ref{main-}, we must arrange the inclusion $\tilde X_1\subset X$ to preserve the homotopy class of the complex structure. Since the space of complex vector space structures on $\R^4$ is simply connected (homotopy equivalent to $S^2$), we can homotope the complex structure $J$ on $X$ to agree with the given Stein structure $\tilde J$ on $\tilde X_1$ near $H$. Let $\tau$ be the standard complex trivialization of $\tilde J|H$ (given by the constant framing in the Legendrian link diagram). We obtain a relative Chern class $c_1(J,\tau)\in H^2(X,H)$. Each 2-handle $h$ of $\tilde X_1$ then has a relative Chern number $c(J)$ for $J$ that only depends on the relative homology class of $h$ in $H_2(X,H)$. Similarly, $h$ has a relative Chern number for the Stein structure $\tilde J$, which is the rotation number $r$ of its attaching circle. In either case, the Chern number reduces mod 2 to the relative Stiefel-Whitney number, which is independent of choice of almost-complex structure, so $r$ is congruent to $c(J)$ mod 2. If we use the above flexibility to choose $h$ differently, its homology class, $c(J)$ and $r$ will not change. However, we can assume there are $\frac12|c(J)-r|$ extra zig-zags that can be flipped over to change $r$ to equal $c_1(J)$. After this change, the relative Chern numbers of $J$ and $\tilde J$ agree on $h$, so the two structures are homotopic rel $H$ over $h$. Applying this procedure to each 2-handle of $\tilde X_1$, we conclude that $J$ and $\tilde J$ are homotopic on $\tilde X_1$. By Theorem~\ref{main-}, we can now isotope $\tilde X_1$ to be a holomorphically embedded Stein handlebody in $X$. Similarly, if $-X$ is complex, we can simultaneously invoke flexibility of the handles in $-\tilde X_2$ to isotope the latter to a Stein handlebody in $-X$. Reinterpreting this as an isotopy of the complex structure rather than $-\tilde X_2$ completes the proof.
\end{proof}

\begin{cor} Every closed, simply connected complex surface $X$ contains a compact, contractible submanifold $A$ with pseudoconcave boundary. We can assume $-A$ admits a Stein domain structure, and if $-X$ is also complex, we can assume after isotopy of its complex structure that $-A$ is pseudoconvex in it.
\end{cor}

\noindent As we have seen (Example~\ref{hsphere}), Theorem~\ref{main-} easily gives pseudoconvex contractible submanifolds of $X$, but pseudoconcavity is harder.

\begin{proof} Akbulut and Matveyev show (without requiring a complex structure) that $X$ decomposes as in Theorem~\ref{AM} with $X_2$ contractible, so $X=\tilde X_1\cup_\partial \tilde X_2$ with $\tilde X_1$ and $-\tilde X_2$ Stein handlebodies, and $A=\tilde X_2$ contractible. The corollary now follows immediately from Theorem~\ref{AMpsc}.
\end{proof}

Now if $Z_1$ and $Z_2$ are closed, oriented, simply connected 4-manifolds with isomorphic intersection pairings, there is a compact, contractible cork $A$ and simply connected complement $Y$ such that $Z_i=Y\cup_{\varphi_i}A$ for $i=1,2$, glued by suitable diffeomorphisms $\varphi_i$ of the boundaries. Akbulut and Matveyev (\cite{AM}~Theorem~5) showed that $Y$ and $-A$ can both be chosen to be Stein handlebodies. We sharpen this in the complex setting:

\begin{thm} \label{cork}
 Suppose that some $Z_i$ is complex. Then
\item[a)] we can assume the corresponding embedded $A\subset Z_i$ has a pseudoconvex boundary. If both $Z_j$ are complex, then the Stein structures on $A$ resulting from the pseudoconvex embeddings are given by the same Legendrian link diagram.
\item[b)]  Alternatively, we can assume $A\subset Z_i$ has a pseudoconcave boundary and Stein complement. If both $Z_j$ are complex, then the resulting Stein structures on $Y$ are given by the same Legendrian link diagram, up to changing the directions of zig-zags. If there is an isomorphism of the intersection forms of $Z_1$ and $Z_2$ that preserves the $Chern$ class, then the Legendrian link diagrams can be assumed to be identical.
\end{thm}

\begin{proof} For (a), apply \cite{AM}~Theorem~5 to $-Z_i$, then observe that $A$ has a unique homotopy class of almost-complex structures, so Theorem~\ref{main-} applies. Proving (b) requires a closer look at the proof of \cite{AM}~Theorem~5. The first step is to construct a simply connected Stein domain $Y'$ that embeds in each $Z_i$ as the complement of some contractible set $A'_i$. Since the Stein structure is constructed as in the proof of Theorem~\ref{AM}, we have enough extra zig-zags to arrange both embeddings $Y'\subset Z_i$ to have pseudoconvex boundaries by Theorem~\ref{main-} when the ambient manifolds are complex. The resulting Legendrian link diagrams then differ only by the directions of these zig-zags. If there is an isomorphism as hypothesized, we can assume it commutes with the inclusions $Y'\subset Z'_i$ (since the construction begins with an arbitrary h-cobordism). Then the two pulled-back complex structures on $Y'$ have the same Chern class, so they are homotopic (since $H^2(Y')$ has no 2-torsion) and the Legendrian diagrams will be identical. The remaining steps of the proof in \cite{AM} change $Y'$ by adding algebraically canceling Legendrian 1-2 pairs, then reconstructing a suitable embedding in each $Z_i$ by working within 4-balls in $Z_i$. These steps cause no changes at the level of homology, so the inclusions still preserve the almost-complex structures up to homotopy, and Theorem~\ref{main-} still applies.
\end{proof}

For related applications, we consider the recent work of Akbulut and Yasui \cite{AY}. Their main Theorem~6.3 asserts that any 4-dimensional, compact, oriented handlebody with $b_2\ne0$ and no 3- or 4-handles can be modified by a homotopy equivalence (preserving the intersection pairing but not the fundamental group of the boundary), after which it admits an arbitrarily large finite number of exotic (nondiffeomorphic) smooth structures, each of which can be realized as a Stein domain. Similarly, they obtain infinitely many such structures when the modified manifold is allowed to be open. In each case, the exotic structures are all made from a single manifold, by twisting on one of a family  of disjointly embedded corks. Since each of their corks explicitly admits an abstract Stein structure, Corollary \ref{contractible} immediately shows that the corks can all be taken to be pseudoconvex domains both before and after the twist (although the ambient complex structure changes radically when we twist). Perhaps more interesting is the connection to pseudoconcave embeddings. Theorem~6.4 of \cite{AY} asserts the following (along with other details.)

\begin{thm} \label{AY} \cite{AY}. Given an embedding $Y\subset Z$ of compact, connected, oriented 4-manifolds (with boundary), suppose the complement $X=Z-\int Y$ has  $b_2\ne0$ and a handle decomposition without 3- or 4-handles. Then for any $n$ there are diffeomorphic submanifolds $Y_1,\dots,Y_n$ in $Z$ such that each $Y_i$ is homotopy equivalent to $Y$ (preserving the intersection pairing but not the fundamental group of the boundary) and the pairs $(Z,Y_i)$ are homeomorphic but not diffeomorphic to each other.
\end{thm}

\noindent That is, after a slight modification of $Y$, it has arbitrarily many smooth embeddings in $Z$ that are topologically equivalent but smoothly different --- a knotting phenomenon that is only visible in the smooth category. We can now sharpen this in the complex setting:

\begin{thm} \label{AYcave}
If $Z$ is closed and complex, then the submanifolds $Y_i$ can be assumed to be pseudoconcave.
\end{thm}

\begin{proof} Akbulut and Yasui prove Theorem~\ref{AY} by applying their main Theorem~6.3 of \cite{AY} to the complement $X$ of $Y$, obtaining Stein domains $X_1,\dots,X_n$ that become the complements of the submanifolds $Y_i$. We combine their proof with that of Theorem~\ref{AMpsc}. The first step of the proof of \cite{AY}~Theorem~6.3 is to slide and order the 2-handles of $X$ so that the first batch is algebraically disjoint from the 1-handles and represents a basis of $H_2(X)$. After this, we insert a step, modifying the handlebody as in the proof of Theorem~\ref{AM} so that it becomes a Stein handlebody smoothly embedded in $Z$, with rotation numbers matching the relative Chern numbers as in the proof of Theorem~\ref{AMpsc} so that we obtain a holomorphically embedded Stein handlebody $\tilde X$ as in the latter proof. (The method of \cite{AM} can be applied to just one side of the splitting, with 3-handles allowed on the other side. The conditions obtained in the first step above are preserved, provided that the construction uses even integers where possible, so that the original 2-handles remain algebraically disjoint from the new 1-handles.) The proof of \cite{AY}~Theorem~6.3 then produces each $X_i$ from $\tilde X$ by ``$W^{\pm}(p)$-modifications'', each depending on a choice of sign and positive integer $p$. Each modification consists of adding a cork (an algebraically canceling 1-2 pair) to the diagram of $\tilde X$ and (for $W^+$) running a 2-handle over the new 1-handle. As noted in \cite{AY}, we still have a Legendrian diagram, but in a $W^+(p)$-modification the diverted 2-handle has $tb$ increased by $p$. We compensate by adding $p$ zig-zags to recover the correct intersection form. Let $H$ consist of the original 0- and 1-handles of $\tilde X$, inside of $X_i$, and let $H'$ be the union of $H$ with the corks from the $W^{\pm}(p)$-modifications. As in the proof of Theorem~\ref{AMpsc}, using contractibility of the corks, we can assume the embedding $X_i\subset Z$ (which was made from that of $\tilde X\subset Z$ by working in disjoint balls) preserves the given almost-complex structures near $H'$. The standard trivialization $\tau$ of the complex bundle structure over $H$ then uniquely extends over $H'$. To apply Theorem~\ref{main-}, it now suffices to verify that the rotation numbers of the 2-handles originating in $\tilde X$ behave properly under the $W^{\pm}(p)$-modifications. It is clear that $W^-(p)$-modifications present no problem, since they change neither the rotation numbers nor relative Chern classes of these 2-handles. In a  $W^+(p)$-modification, however, there is one 2-handle $h$ diverted over the 1-handle of the new cork. Let $h^*$ be the 2-handle of the new cork. Under the obvious isomorphism $H_2(X_i,H)\to H_2(X_i,H')$, the relative class of $h$ pulls back to the class of the relative cycle $h-ph^*$. The Legendrian attaching circle of $h^*$ has rotation number $\pm1$, but that of $h$ had been modified by adding $p$ zig-zags. We choose the signs of the latter so that the relative Chern class of the handlebody on $h-ph^*$ agrees with its value on $h$ before the modification. Then the embedding preserves the homotopy class of the almost-complex structures everywhere, and Theorem~\ref{main-} isotopes the embedding so that its boundary is pseudoconvex and its complement $Y_i$ is pseudoconcave as required. The rest of the proof in \cite{AY}, analyzing the topology of these examples, is unchanged, except that $\tilde X$ and its complement have been substituted for $X$ and $Y$. Note that Akbulut and Yasui require the rotation numbers to have different values from ours, in order to construct Stein structures on $X_i$ that constrain the minimal genera of homology classes, guaranteeing different diffeomorphism types. However, those complex structures need not be related to that of $Z$. We just observe that the diffeomorphism type of $X_i$ does not change when we change the direction of zig-zags to modify the relative Chern class.
\end{proof}


\section{Background: cores, convexity and contact manifolds} \label{Back}

Having explored various applications, we now wish to prove the main theorems, suitably generalized. In preparation, we review some basic background material. For more detail, see \cite{CE}, \cite{OS}. We first return to Eliashberg's construction of Stein manifolds. Section~\ref{basics} discussed how to attach a single handle, preserving pseudoconvexity of a boundary. Eliashberg uses this to inductively construct Stein manifolds. (In his approach, a trivial collar of the boundary is also attached with each handle, so infinite topology does not cause difficulties.) To prove that the resulting manifold is Stein, one must see that it admits an exhausting, plurisubharmonic function; this suffices by work of Grauert \cite{Gr}. Plurisubharmonic (equivalently, {\em $J$-convex}) functions have various useful properties \cite{CE}. The maximum of two such functions is a ``continuous $J$-convex function'', which can be smoothed to a plurisubharmonic function by a $C^0$-small perturbation, relative to a compact subset near which the function is already smooth. 
Plurisubharmonicity and pseudoconvexity are both local, $C^2$-open conditions. Thus, every pseudoconvex hypersurface $M$ in a complex manifold has a neighborhood foliated by such hypersurfaces, given as a suitable open subset of any tubular neighborhood $\R\times M$ with its product foliation. One can then introduce a plurisubharmonic function: Any function with no critical points and all level sets pseudoconvex (such as projection to $\R$ in the above neighborhood) can be made plurisubharmonic on a preassigned compact subset of its domain by postcomposing with a suitable diffeomorphism of $\R$. 
(It then increases in the outward direction, so the level sets are pseudoconvex in the boundary orientation from the sublevel sets.) This can be done relative to some open subset on which plurisubharmonicity already holds (and where critical points are allowed), since plurisubharmonicity is preserved when we postcompose with any smooth function $f$ with $f'>0$ and $f''\ge 0$. 
For example, we can prove that a given compact complex manifold $W$ is a Stein domain if we can just construct a plurisubbharmonic function $\varphi\co W\to [0,a]$ with $\partial W=\varphi^{-1}(a)$, and a codimension-0 holomorphic embedding of $W$ into some open complex manifold $X$. Then we extend $\varphi$ over a collar of $\partial W\subset X$ to get an exhausting plurisubharmonic function realizing a neighborhood of $W$ as a Stein manifold with $W$ a sublevel set. (Beware that an arbitrary Stein domain in a Stein manifold need not arise as a plurisubharmonic sublevel set of the latter: $S^1\times S^1\subset \C\times \C$ has a tubular neighborhood that is a Stein domain, but any relative Morse function on the pair must have an index-3 critical point to kill $H_2(T^2)$.) Eliashberg uses this technology to construct the required plurisubharmonic function on his manifold: Suppose a $k$-handle $h$ is attached by his method to $W$ along a pseudoconvex boundary, and $\varphi$ is a plurisubharmonic function on $W$ with $\varphi^{-1}(a)=\partial W$. Eliashberg shows that after composition with a suitable $f$ on $(-\infty,a]$ that is the identity outside a preassigned neighborhood of $a$, and after modification of $\varphi$ in a preassigned neighborhood of $h$, the function extends plurisubharmonically over $h$ with a unique new critical point (Morse, of index $k$) so that $\partial (W\cup h)$ is a level set. Induction on handles gives the required plurisubharmonic function on Eliashberg's manifold.

Recall that the extra complication of Eliashberg's method in complex dimension 2 is that a 2-handle core cannot always be made totally real by an isotopy suitably controlled on the boundary. More generally, suppose $F$ is a compact, connected, oriented surface in a complex surface $X$. Generically, $F$ is totally real in a neighborhood of its boundary. This condition gives a canonical complex basis $(\tau,n)$ for $TX$ along $\partial F$, namely the tangent and normal vector fields to $\partial F$ in $TF|\partial F$, which we can use to define relative characteristic classes. Eliashberg and Harlamov (\cite{EH}, see \cite{N} for a proof in English) showed that $F$ is then isotopic, rel a neighborhood of the boundary, to a totally real surface if and only if two relative characteristic numbers vanish. The first is $e(\nu F)+\chi (F)$, where $\chi(F)$ is the Euler characteristic, and $e(\nu F)$ is the normal Euler number of $F$ relative to $J\tau$, only depending on $J$ through $J\tau$. The second is $\langle c_1(TX,\tau,n),F\rangle$, which only depends on the complex bundle $TX$, its framing over $\partial F$, and the class of $F$ in $H_2(X,\partial F)$. The isotopy is $C^0$-small. If $M\subset X$ is an oriented hypersurface, we have seen that it inherits an oriented plane field $\xi$. In fact, the canonical map from complex bundle structures on $TX|M$ to oriented plane fields is a bijection on homotopy classes, preserving $c_1$, with inverse constructed by choosing a complementary trivial complex line subbundle. If a compact, oriented surface $F$ intersects $M$ transversely along $\partial F=M\cap F$, and this is Legendrian in $M$, then $F$ must be totally real near $\partial F$. Thus, the two obstructions to reality are defined, and are preserved by any homotopy of the almost-complex structure that fixes $\xi$. The vector field $J\tau$ on $\partial F$ gives its {\em contact framing}, which on any Legendrian link is defined by a vector field in $\xi$ nowhere tangent to the link. In a Legendrian diagram, the  invariant $tb$ is the coefficient of the contact framing, so for an almost-complex 2-handle attached along a Legendrian knot, the reality obstructions for its core are measured by the difference of its framing and relative Chern class from $tb-1$ and $r$, respectively.  In general, Eliashberg's method attaches a 2-handle along a Legendrian knot with framing obtained from the contact framing by a left twist.

 We now consider contact 3-manifolds in more detail. A {\em contactomorphism} is a diffeomorphism preserving contact structures, and a {\em contact embedding} is defined similarly. On a manifold $M$ (without boundary), Gray's Theorem \cite{Gray} asserts that two contact structures related by a compactly supported homotopy (1-parameter family) $\xi_t$ of contact structures are related by an {\em isotopy}, i.e., a contactomorphism ambiently isotopic to the identity. The isotopy varies smoothly with respect to any auxiliary parameters, is real-analytic wherever $M$ and $\xi_t$ are, and fixes any point at which $\xi_t$ is independent of $t$ (so has compact support). Gray's Theorem can be used to show that every {\em Legendrian isotopy} in $M$ extends to a {\em contact isotopy}. That is, given an isotopy of a Legendrian link $L$ through a family of Legendrian links, there is a (compactly supported) isotopy of $\id_M$ through contactomorphisms that, when composed with the inclusion of $L$, returns the original Legendrian isotopy.

 The most interesting contact 3-manifolds are {\em tight}. The classification of tight contact structures is a delicate, and in general unsolved, problem, whereas every homotopy class of plane fields on a closed, oriented 3-manifold contains an essentially unique contact structure failing to be tight \cite{overtwisted}. A contact 3-manifold is {\em tight} if no disk in it can have Legendrian boundary with contact framing homotopic to the outward normal of the disk. Equivalently, tightness means that every unknotted Legendrian circle has $tb<0$ (where the latter is defined since framing coefficients are well-defined for nullhomologous knots, without reference to a preassigned diagram). The boundary of a Stein domain must be tight \cite{Durham}, and an open contact 3-manifold is clearly tight whenever every open subset with compact closure is tight. We will deduce from these facts (Proposition~\ref{tight}(b)) that every contact 3-manifold bounding a Stein shard is tight.

 Tight contact structures on some simple contact 3-manifolds (with boundary) have been classified. For example, $S^3$ and $S^2\times I$ admit unique tight contact structures up to isotopy, relative to a neighborhood of the boundary with a preassigned tight structure in the latter case \cite{ball}. Tight contact structures on the solid torus $T=S^1\times D^2$ have also been classified \cite{Gi}, \cite{H}. For simplicity we assume that the boundary of $T$ is {\em convex}, as is generic for closed, oriented surfaces $F$ in contact 3-manifolds. This means we can identify a neighborhood of $F$ with a neighborhood of $\{0\}\times F$ in $\R\times F$ endowed with some $\R$-invariant contact structure. We then have a {\em dividing set}, which is the subset of $F$ on which $\xi$ is parallel to the $\R$-factor. This is always a smoothly embedded, compact 1-manifold separating $F$ into positive and negative regions (since we continue to assume $\xi$ is oriented). In a tight contact 3-manifold, the dividing set of a convex torus is always a nonzero, even number of parallel essential circles. While the classification of tight, convex solid tori is somewhat complicated, we only need the simplest case (e.g.\ \cite{OS} Theorem~5.1.30): When the dividing set in $\partial T$ is a pair of {\em longitudes}, i.e., circles intersecting $\{1\}\times \partial D^2$ algebraically once, then the contact structure is unique up to isotopy, relative to a fixed structure near the boundary. (It is immediate from the proof that the isotopy fixes a neighborhood of the boundary.) An explicit model is given by a closed, cylindrical neighborhood $T_0$ of the $y$-axis in $\R^3$ modulo unit $y$-translation, with contact form $dz+xdy$. The $\R$-direction near $\partial T_0$ is radial, with $\R$-action multiplying $x$ and $z$ by $e^t$, and the two dividing curves are longitudes following the (constant) contact framing on the core circle $K$. Any Legendrian knot has a standard tubular neighborhood contactomorphic to this $T_0$. Every Legendrian circle $K'$ in $T_0$ that is smoothly isotopic to the core $K$, preserving the contact framings, is Legendrian isotopic to it. (To see this, note that $K$ is Legendrian isotopic to the circles in $\partial T_0$ directly above and below it. Now any such $K'$ in $T_0$ can be pushed into $\int T_0$ by convexity. Since its contact framing agrees with that of $K$, its standard tubular neighborhood $T'\subset\int T_0$ has dividing curves on its boundary that are parallel to those of $T_0$. Classification theory then shows that $T_0-\int T'$ is contactomorphic rel $\partial T_0$ to a cylindrical shell around $K$, so inspection shows that corresponding Legendrian circles in the two boundary components of $T_0-\int T'$ are Legendrian isotopic.) Similarly, any tight, convex solid torus $T_1$ divided by a pair of longitudes contains a unique Legendrian isotopy class of Legendrian cores with contact framing determined by the dividing curves. This is because Giroux's Flexibility Theorem guarantees that $T_1$ is isotopic within itself to a copy of $T_0$, whose boundary lies with degree 1 in a preassigned neighborhood of $\partial T_1$, and whose dividing set is parallel to that of $T_1$. Note that other Legendrian cores for $T_1$ can be created by adding zig-zags, but their contact framings will have left twists relative to the dividing curves.


\section{Proof of the main theorem} \label{Smooth}

The goal of this section is to present and prove the main principle in its full generality as Theorem~\ref{main}, and deduce various consequences including Theorems~\ref{smooth} and \ref{main-}. We begin by carefully establishing our conventions for handlebodies, and expanding to the relative case. We will then extend Stein theory in arbitrary dimensions to allow {\em Stein shards}, noncompact manifolds with boundary, culminating in a notion of (possibly noncompact) {\em relative Stein handlebodies}. While we do not need the full generality of Theorem~\ref{main} for the applications in the present paper, it will be useful in the sequel \cite{steintop}.

To attach a $k$-handle $h$ to an $m$-manifold $W$ (with boundary), we first glue $D^k\times D^{m-k}$ onto $W$ by a diffeomorphic embedding of the attaching region $\partial D^k\times D^{m-k}$ into $\partial W$. Then we smooth the corners by removing a neighborhood of the co-attaching region $\int D^k\times\partial D^{m-k}$ disjoint from the attaching region, creating a smooth boundary. (A neighborhood of the corner locus is modeled by $\partial D^k\times\partial D^{m-k}$ crossed with $\R^2$ minus the open first quadrant. Replace the $C^0$ boundary curve of this last factor by a $C^\infty$ curve agreeing with the positive $x$-axis, but curving upward parallel to the $y$-axis for negative $x$. Thus, the boundaries of $W$ and $W\cup h$ are $C^\infty$-close where they join along the boundary of the attaching region.)

\begin{de} For an $m$-manifold $W$ with boundary (not necessarily compact), a {\em handlebody relative to} $W$ is a nested collection of $m$-manifolds $W=H_{-1}\subset H_0\subset H_1\subset \dots\subset H_m=H$, where for each $k\ge0$, $H_k$ is made from $H_{k-1}$ by attaching (possibly infinitely many) $k$-handles with disjoint attaching regions. A {\em handlebody} is the special case when $W$ is empty.
\end{de}

\noindent Thus, handles are attached in order of increasing index. Since we never add collars to the boundaries, there can only be finitely many handles if $\partial H_0$ is compact. (Otherwise, clustering would destroy the manifold with boundary.) To allow infinite topology in this context, we are then forced to allow noncompact boundaries. We next adapt Stein theory to this setting.

\begin{de} \label{shard}
A complex manifold $W$ with pseudoconvex boundary (not necessarily compact) will be called a {\em Stein shard} if it has a proper, codimension-0 holomorphic embedding into a Stein manifold.
\end{de}

\noindent Properness guarantees that the image of the embedding is a closed subset, so Stein shards without boundary are precisely Stein manifolds. A Stein shard necessarily has an exhausting plurisubharmonic function $\psi\co W\to[0,\infty)$ (where we specify no additional boundary condition for $\psi$) obtained by restricting one on the ambient Stein manifold. Corollary~\ref{existspsh} below gives a converse.

The next two propositions are easy consequences of Grauert's characterization of Stein manifolds via plurisubharmonic functions \cite{Gr}. We present them to elucidate Stein shards and pseudoconcave fillings, and to introduce the (classical) method of proof, which relies heavily on the facts gathered in the first paragraph of the previous section.

\begin{prop}\label{int} The interior of a Stein shard is a Stein manifold.
\end{prop}

\begin{proof}  Given a Stein shard $W$, let $U\subset W$ be a neighborhood of $\partial W$, admitting a smooth function $\varphi_0\co U\to \R$ without critical points, such that $\partial W$ is its maximal level set $\varphi_0^{-1}(a)$ and all level sets are pseudoconvex. Let $\{K_i\}$ be a sequence of compact subsets of $U$ for which $U^*=\bigcup\int K_i$ contains $\partial W$ but $\varphi_0(K_i)\subset [a-\frac1i,a]$ for each $i$. Inductively construct a sequence of functions $\varphi_i$ on $U$, with each $\varphi_i$ plurisubharmonic on $K_1\cup\dots\cup K_i$, creating $\varphi_i$ from $\varphi_{i-1}$ for each $i>0$ by postcomposing with a suitable function. We can assume that each point of $U^*-\partial W$ has a neighborhood on which the sequence is eventually constant, so the sequence converges to a plurisubharmonic $\varphi^*\co U^*-\partial W\to [0,\infty)$ that we can assume approaches $\infty$ along $\partial W$. Now let $M\subset U^*$ be a hypersurface separating $\partial W$ from $W-U$. Let $\psi$ be an exhausting plurisubharmonic function on the Stein shard $W$. By postcomposing $\psi$ with a suitable function, we can arrange $\psi>\varphi^*$ everywhere along $M$.  Then the function $\max(\varphi^*,\psi)$ on the region between $M$ and $\partial W$, extended by $\psi$ over the rest of $\int W$, is a continuous $J$-convex function on $\int W$, so it smooths to an exhausting plurisubharmonic function exhibiting $\int W$ as a Stein manifold. 
\end{proof}

\begin{prop}\label{tight}
(a) A connected complex manifold $W$ with  nonempty compact boundary and complex dimension $>1$ is a Stein shard if and only if it is a Stein domain.

\item{(b)} Every compact subset of a Stein shard $W$ holomorphically embeds rel $\partial W$ in a Stein domain of the same dimension as $W$. In particular, every compact subset of $\partial W$ has a contact embedding in a Stein boundary, so when $\dim_\C W=2$, $\partial W$ is a tight contact 3-manifold.

    \item{(c)} A compact, connected pseudoconvex hypersurface in a Stein manifold $V$ with complex dimension $>1$ must be oriented outward, i.e., as the boundary of a Stein domain in $V$.
\end{prop}

\begin{proof} For (b), $W$ is a Stein shard, so we are given a holomorphic embedding $W\subset V$ into a Stein manifold with an exhausting plurisubharmonic $\psi\co V\to  [0,\infty)$. Suppose $\psi$ maps the given compact subset $K$ into $[0,a)$. Find an embedding $[0,2]\times N\subset V$, where $\{1\}\times N\subset\partial W$ is some compact submanifold with boundary, containing $\partial W\cap \psi^{-1}[0,a]$ in its interior. Since $\psi|(\{1\}\times\partial N)>a$, we can assume, after mapping $[0,2]$ into a smaller neighborhood of 1, that $\psi|([0,2]\times\partial N)>a$. We can also arrange $(1,2]\times N$ to be disjoint from $W$, so that $W\cap\psi^{-1}[0,a]$ intersects $\partial ([0,2]\times N)$ only inside $\{0\}\times N$. After further shrinking the interval, each level $\{t\}\times N$ is pseudoconvex, so we can find a plurisubharmonic $\varphi\co [0,2]\times N\to[-1,\infty)$ with these level sets. Rescaling and adding a constant if necessary, we can assume  $\varphi^{-1}(a)=\{1\}\times N$ and $\varphi^{-1}(-1)=\{0\}\times N$. The function $\max(\varphi,\psi)$ extends by $\psi$ over a neighborhood of $W\cap\psi^{-1}[0,a]$ in $V$, and can then be smoothed rel $\partial W\cap K$ to a plurisubharmonic $\eta$ for which $a$ is a regular value. Then $\eta^{-1}[0,a]$ is the required Stein domain. For (c), observe that the given hypersurface $N\subset V$ must bound a compact region in $V$, for otherwise it would represent a nontrivial homology class above the middle dimension when $\dim_\C(V)>1$. Suppose $N$ is oriented inward. Construct $\varphi$ on a collar of $N$ as before, approaching infinity on the inside boundary and $-1$ on the outside boundary. Splicing this into an exhausting plurisubharmonic function on $V$ as before exhibits an open subset of $V$ as a Stein manifold separated into two noncompact regions by $N$, violating our previous observation. Thus, $N$ is oriented outward, and the same argument with  $\varphi$ reaching $-1$ on the inside realizes $N$ as the boundary of a Stein domain. Now (c) implies that for a Stein shard $W\subset V$ as in (a), $\partial W$ is connected and $W$ is the Stein domain cut out of $V$ by $\partial W$.
\end{proof}

\begin{de} \label{SHB} A complex manifold with boundary, exhibited as a handlebody $H$ relative to a submanifold $W$, will be called a {\em relative Stein handlebody} if
 \item{a)} each $H_k$ ($k\ge -1$) is a Stein shard, and
 \item{b)}  when $\dim_\C H=2$, the attaching region of each 2-handle, which is a solid torus in $\partial H_1$, has convex boundary with  two dividing curves, each a longitude that when modified by a left twist determines the 2-handle framing.
\end{de}

Condition (b) guarantees cores of the sort arising in Eliashberg's construction:

 \begin{prop} \label{core} Each $k$-handle of a relative Stein handlebody has a totally real core disk that is $J$-orthogonal to $\partial H_{k-1}$. In particular, its attaching sphere is isotropic in $\partial H_{k-1}$.
 \end{prop}

 \begin{proof} Such disks are guaranteed by the general theory \cite{E}, \cite{CE} except when $\dim_\C H=k=2$. In that case, the attaching region $T \subset \partial H_1$ of the given 2-handle $h$ inherits a convex contact structure guaranteeing an essentially unique Legendrian attaching circle $K$ with the correct contact framing (giving the 2-handle framing after a left twist). We can assume $K$ is real analytic (cf.\ proof of Theorem~\ref{main}). The boundary $\partial T$ also bounds the co-attaching region  $T'= \partial H_2\cap h$. The nontrivial circle in $\partial T$ that bounds in $T'$ gives the framing of $h$, so it intersects each dividing curve once. Thus, the contact structure on $T'$, which is tight by Proposition~\ref{tight}(b), is determined by its restriction to a neighborhood of $\partial T$ in $\partial H_2$. Note that $T\cup T'=\partial h$ is homeomorphic to $S^3$. Now we construct a model by abstractly attaching a 2-handle $h_0$ to $H_1$ inside $T$, using Eliashberg's method, along the given $K$. We get a solid torus $T_0'$ in the new boundary as before, contactomorphic to $T'$ rel boundary. It suffices to show that the core disks of $h$ and $h_0$ have the same Eliashberg-Harlamov reality obstructions, since $h_0$ has a totally real core and vanishing invariants by construction. But the normal Euler numbers agree since $h$ and $h_0$ are homeomorphic rel $K$ and its contact framing. For the relative Chern number, note that $H_2(h,K)\cong H_1(K)\cong H_2(T\cup T',K)$, so the homology class of the core disk is carried by $T\cup T'$ and determined there by $K$. The complex structure on the bundle $TH|(T\cup T')$ is determined up to homotopy by the contact structures on $T$ and $T'$, so the homotopy equivalence from $T\cup T'$ to $T\cup T_0'$ rel a neighborhood of $K$ preserves the relative Chern number as required.
 \end{proof}

\begin{prop}\label{SteinHB} If a relative handlebody $H$ is made from a Stein shard $W$ (or, more generally, from any relative Stein subhandlebody) by attaching handles by Eliashberg's method, then we can assume it is a relative Stein handlebody.
\end{prop}

\noindent The only change we make in $H$ is an adjustment of the 2-handles when $\dim_\C H=2$; when this does not occur, we prove that $H$ is already a relative Stein handlebody. Recall that we use the term ``Stein handlebody" in this paper as a precise alternative to the intrinsically vague condition ``made by Eliashberg's method". In complex dimension 2, Eliashberg's method constructs Stein handlebodies (with $W=\emptyset$ and finitely many handles) from Legendrian link diagrams. Conversely, every compact Stein handlebody in this dimension determines such a diagram (up to suitable moves \cite{Ann}, \cite{GS}): Definition~\ref{SHB}(b) guarantees that the 2-handle attaching regions in $\partial H_1$ have essentially unique Legendrian core circles whose contact framings follow the dividing curves, determining a Legendrian link in $\partial H_1$ up to Legendrian isotopy.

\begin{proof} We assume $H$ is made by Eliashberg's method from $W$. For the general case, incorporate the Stein subhandlebody into $W$ (and hence into each $H_k$) without changing notation.

 To prove (b) of the definition when it applies, note that by construction, each 2-handle is attached to a neighborhood $N$ of a Legendrian circle $K$ in $\partial H_1$, with framing obtained from the contact framing of $K$ by a left twist. This $N$ can be taken to lie in a standard neighborhood $T$ of $K$, since the 2-handles form the top layer and so can be thinned as needed. After a further $C^\infty$-small perturbation of $\partial H_2$, splitting it away from $\partial H_1$ over $T$, we can assume $N=T$. (This is the only part of the proof where we need to adjust $H$.) The dividing curves of $T$ are positioned correctly since they follow the contact framing of $K$.

To prove (a), inductively suppose that $H_{k-1}$ is a Stein shard, as is true by hypothesis for $H_{-1}=W$. Since $H_k$ is made from $H_{k-1}$ by attaching a layer of (possibly infinitely many) disjoint handles by Eliashberg's method, $\partial H_k$ is automatically pseudoconvex. Thus, it suffices to properly embed $H_k$ in a Stein manifold. We can at least properly embed $H_k$ in an open complex manifold $V$ by starting with a neighborhood of $H_{k-1}$ in its given Stein manifold, then using the extension implicit in Eliashberg's construction. We wish to cut $V$ down to a Stein manifold using a plurisubharmonic function constructed as in Proposition~\ref{int}. Start as before in $H_{k-1}$, with a function $\varphi_0\co \cl (U)\to \R$ without critical points, with level set $\varphi_0^{-1}(a_0)=\partial H_{k-1}\subset U\subset H_{k-1}$, and all levels pseudoconvex. Let $\{K_i\}$ be a sequence of compact subsets of $H_k$, disjoint from $H_{k-1}-U$, for which $\bigcup\int K_i$ (taking interiors in $H_k$) contains $\partial H_{k-1}$, but $\varphi_0(K_i\cap H_{k-1})\subset [a_0-\frac1i,a_0]$ for each $i$. Also assume that each $k$-handle in $H_k$ lies in every $K_i$ that it intersects, and that the subsets $K_i\cap\partial H_{k-1}$ are nested. After postcomposing with a suitable function, we can assume $\varphi_0$ is plurisubharmonic near $K_1\cap H_{k-1}$.  Since $K_1$ contains only finitely many $k$-handles, Eliashberg's method further modifies $\varphi_0$, rel any preassigned compact subset of $\int H_{k-1}$, to $\hat\varphi_1\co U\cup K_1\to \R$, plurisubharmonic near $K_1$, with level set $\hat\varphi_1^{-1}(a_1)=\partial (U\cup K_1)$ for some $a_1\ge a_0+1$. To begin extending to the required Stein neighborhood of $H_k$ in $V$, we enlarge the domain of  $\hat\varphi_1$ slightly: Let $f_t$ be a smooth ambient isotopy of $V$, compactly supported away from $H_{k-1}-K_2$ and the $k$-handles outside $K_1$, expanding $K_1$ so that $f_t(K_1\cap\partial H_k)$ is disjoint from $H_k$ for all $t>0$. Compact support guarantees that for sufficiently small $t>0$, the map $\varphi_1=\hat\varphi_1\circ f_t^{-1}$ on $f_t(U\cup K_1)$ is still plurisubharmonic near $f_t(K_1)$. Now inductively repeat the entire construction for each $K_i$, leaving the part of $\varphi_{i-1}$ over $K_{i-1}$ fixed except for postcomposing as necessary, to obtain a sequence of functions $\varphi_i$, defined and plurisubharmonic in a neighborhood of $K_1\cup\dots\cup K_i$ in $V$. The sequence is eventually constant on a neighborhood of any given point where it is eventually defined, except for boundary points where the sequence increases without bound. The resulting limit $\varphi^*\co U^*\to [0,\infty)$, on an open subset $U^*$ of $V$ given by $\bigcup\int K_i$ union a neighborhood of $\partial H_k$, is plurisubharmonic and approaches infinity along a pseudoconvex hypersurface in $V-H_k$ parallel to $\partial H_k$.

To exhibit $V^*=U^*\cup H_{k-1}$ as the required Stein manifold containing $H_k$, let $M\subset U^*\cap \int H_{k-1}$ be a hypersurface separating $\partial H_{k-1}$ from $H_{k-1}-U$ in $V$. We wish to construct an exhausting plurisubharmonic function $\psi$ on $H_{k-1}$ with $\psi>\varphi^*$ on $M$ and $\varphi^*>\psi$ on $\partial H_{k-1}$, for then $\max(\varphi^*, \psi)$ will extend and smooth to the required exhausting plurisubharmonic function on $V^*$. We construct $\psi$ concurrently with $\varphi^*$. Begin with any exhausting plurisubharmonic function $\psi_0$ on the Stein shard $H_{k-1}$; this will become $\psi$ after suitable postcomposition. Now $M$ is the nested union of compact subsets $M_i=M\cap P_i$, where $P_i= \psi_0^{-1}(-\infty,b_i]$ for some $b_i\ge \max\psi_0|K_i\cap H_{k-1}$ (so $K_i\cap H_{k-1}\subset P_i$). Since postcomposing preserves level sets, and $\varphi_0|M_1<a_0$, we can assume $\varphi_0$ is plurisubharmonic on the compact set $\bigcup_i K_i\cap\varphi_0^{-1}(-\infty,c_1]$ for some $c_1\ge \max\varphi_0|M_1$. Postcompose $\psi_0$ with a suitable function to get $\psi_1>\varphi_0$ on $M_1$. Since $M_1$ is a compact subset of $\int H_{k-1}$, we can control the construction of $\varphi_1$ so that it agrees with $\varphi_0$ on $M_1$ (hence $\psi_1>\varphi_1$   on $M_1$) but $\varphi_1>\psi_1$ on the compact set $K_1\cap \partial H_{k-1}$. We can also assume $\varphi_1$ is plurisubharmonic on $\bigcup_i K_i\cap\varphi_1^{-1}(-\infty,c_2]$ for some $c_2\ge \max\varphi_1|M_2$ (using the analogous condition previously obtained for $\varphi_0$). Postcomposing with a suitable function, we can now modify $\psi_1$ rel its sublevel set $P_1$ to get $\psi_2$ with  $\psi_2>\varphi_1$ on $M_2$. Since $\varphi_1>\psi_1=\psi_2$ on $K_1\cap \partial H_{k-1}\subset P_1$, we can now construct $\varphi_2$ so that it agrees with $\varphi_1$ on $M_2$ (in fact on their sublevel set $\varphi_1^{-1}(-\infty,c_2]$) but exceeds $\psi_2$ on $K_2\cap \partial H_{k-1}$.  Continuing in this manner, we construct sequences $\varphi_i$ and $\psi_i$ simultaneously so that the limits satisfy the required inequalities.
\end{proof}

Note that the Stein shard condition was only used for three things --- to guarantee a pseudoconvex boundary on $H_{k-1}$, an exhausting plurisubharmonic function on it, and a codimension-0 embedding of it in some open complex manifold. (We can always arrange such an embedding to be proper by cutting down the ambient manifold.) Thus, the trivial case $H=W$ shows:

 \begin{cor}\label{existspsh}
 A complex manifold with pseudoconvex boundary is a Stein shard if and only if it admits an exhausting plurisubharmonic function and a codimension-0 holomorphic embedding into some complex manifold (without boundary).\qed
\end{cor}

We can now state and prove the main theorem, a generalization of Theorem~\ref{main-}.

\begin{thm}\label{main}
Let $H$ be a Stein handlebody relative to $W$, with $\dim_\C H=2$, and let $f\co H\to X$ be a smooth embedding into a complex surface. Suppose that the complex structure $J_X$ on $X$ pulls back to one on $H$ that is homotopic rel $W$ (through almost-complex structures) to the original complex structure $J_H$. Then $f$ is smoothly isotopic (ambiently if $f$ is proper) to $\hat f$, with $\hat f(H)$ exhibited as a Stein handlebody relative to $\hat f(W)$ (for the complex structure inherited from $X$ and the handle structure inherited from $H$). Furthermore, $\hat f$ restricts to a contactomorphism from each $\partial H_k$ to its image. Each map in the isotopy sends each $H_k$ into $f(H_k)$ (and similarly for every subhandlebody of $H$). The isotopy on $W$ is $C^r$-small (for any preassigned $r\ge 2$) and supported in a preassigned neighborhood of the attaching regions.
\end{thm}

\noindent It seems reasonable (although vague) to assert that $\hat f(H)$ is made by Eliashberg's method (at least up to some $C^r$-small perturbations). In any case, if $H$ is compact with $W=\emptyset$, then $H$ and $\hat f(H)$ determine the same Legendrian handle diagram, since $\hat f$ preserves the contact structure on $\partial H_1$ (cf.\ discussion following Proposition~\ref{SteinHB}). Theorem~\ref{main-} follows immediately.

\begin{proof} We can reduce to the case that $f$ is proper, by replacing $X$ by a neighborhood of $f(H)$ consisting of $f(H)$ and an open collar $(0,1)\times f(\partial H)$ of its boundary. The isotopy we construct will be ambient in this reduced $X$, and fix the outer half $(\frac12,1)\times f(\partial H)$ of the collar. For $k=0,1,2$, we inductively assume that $f$ has already been isotoped as in the theorem so that $f|H_{k-1}$ satisfies the conclusions given for $\hat f$, that is, $f(H_{k-1})$ is a relative Stein handlebody to whose subhandlebody boundaries $f$ is a contactomorphism. We then attempt to extend these conclusions over $f|H_k$. We also inductively assume that $f^*J_X$ is homotopic to $J_H$ on $H-\int H_{k-1}$ through almost-complex structures inducing a fixed contact structure on $\partial H_{k-1}$. (This latter constraint along $\partial H_{k-1}$ is the key to controlling the 2-handles.) Note that the induction hypotheses are trivially satisfied for the original $f$ when $k=0$, since the hypotheses of the theorem guarantee that $f|W$ is holomorphic.

To recover the induction hypotheses over $H_k$, we ambiently isotope $f$ near each $k$-handle $h$. By Proposition~\ref{core}, $h$ has a totally real core $D$ with $\partial D\subset \partial H_{k-1}$ isotropic (so Legendrian when $k=2$). We arrange $f(\partial D)$ to be a real-analytic submanifold of $X$. This is trivially true unless $k=2$, in which case we first isotope $f$ rel $\partial H_k$, pushing part of $f(\partial H_{k-1})$ inward into $f(H_{k-1})$ and making it an analytic submanifold near $f(\partial D)$. We make the isotopy sufficiently $C^r$-small, $r\ge 2$, that $f(\partial H_{k-1})$ remains pseudoconvex, so by Gray's Theorem we can assume $f|\partial H_{k-1}$ remains a contactomorphism. Now by Lemma 2.5.1 of \cite{E}, we can further perturb $f$, contactomorphically on $f(\partial H_{k-1})$, so that $f(\partial D)$ is real-analytic but still Legendrian. (First find an analytic curve $K$ $C^r$-close to $f(\partial D)$, then analytically perturb the contact form to make $K$ Legendrian. Gray's Theorem gives an isotopy from $K$ to an analytic, Legendrian curve in the original contact structure, contact isotopic to $f(\partial D)$.) We can assume our hypotheses are still intact.  (The only potential difficulty is that $f$ on $\partial H_0-\partial H_1$  may have been perturbed in a $C^r$-small way near where $\partial D$ crosses onto a 1-handle. To make $f$ again a contactomorphism there without disturbing our work on $\partial H_1$, find a 1-parameter family of hypersurfaces locally interpolating between $\partial H_0$ and $\partial H_1$ there, then apply Gray's construction. This varies smoothly with respect to the parameter and fixes $\partial H_1$ since $f$ is already a contactomorphism there.) For any $k$, we now isotope $D\subset H$ rel boundary (fixing $f$) so that $f(D)$ has an analytic 1-jet along $f(\partial D)$ and is $J_X$-orthogonal to $f(\partial H_{k-1})$. We further arrange $f(D)$ to be totally real. Again, this is trivially true unless $k=2$. In that case, we recall that $D\subset H$ was totally real with Legendrian boundary, so the reality obstructions of Eliashberg and Harlamov vanish for $J_H$ on $D$. Since $J_H$ is homotopic to $f^*J_X$ on $H-\int H_{k-1}$, fixing the induced contact structure on $\partial H_{k-1}$, the obstructions for $f^*J_X$ also vanish on $D$, so there is an isotopy of $D$ rel its 1-jet along  $\partial H_{k-1}$ making $f(D)$ totally real. For any $k$, a further $C^r$-small isotopy makes $f(D)$ analytic but still totally real and intersecting  $f(\partial H_{k-1})$ $J_X$-orthogonally along $f(\partial D)$. Eliashberg's main lemma now locates a $k$-handle $\hat h\subset f(h)$ with core $f(D)$, attached to $f(H_{k-1})$ preserving pseudoconvexity of the boundary. By Proposition~\ref{SteinHB}, the union $\hat H_k$ of $f(H_{k-1})$ with the handles $\hat h$ for all $k$-handles of $H$ is a relative Stein handlebody. We now construct an isotopy of $f$, supported near the $k$-handles and $C^r$-small on $H_{k-1}$, after which $f$ maps the set $H_k$ onto $\hat H_k$. First, push each $k$-handle $f(h)$ down (toward $f(H_{k-1})$) so that its attaching region $f(\partial_-h)$ has fixed boundary but interior disjoint from its former location, and near the boundary of $\partial_-h$, $f|\partial H_k$ goes contactomorphically into $\partial \hat H_k$. (This is $C^r$-small since $\partial H_k$ is $C^{r+2}$-close to $\partial H_{k-1}$, on which $f$ was already a contactomorphism.) We can assume that $f$ is still a contactomorphism on each $\partial H_i$, $i<k$, by interpolating between boundaries and applying Gray's construction as before. It is now easy to isotope $f$ rel a neighborhood of $H_{k-1}$, so that $f(H_k)=\hat H_k$. Our initial induction hypotheses are preserved, and we have incremented $k$ for one of them: $f$ induces a Stein handlebody structure on $f(H_k)$.

There are two induction hypotheses left to reconstruct: We must arrange $f|\partial H_k$ to be a contactomorphism, and construct a suitable homotopy of $f^*J_X$ if $k<2$. The first hypothesis is easy when $k=0$, since the boundary of a 0-handle is $S^3$, on which any two tight contact structures are isotopic. (Note that the handle lies in the interior of its original location, so we can perform the isotopy while keeping each subhandlebody of $f(H)$ within its original location as required.) When $k=2$, the co-attaching region $\partial_+ h=\cl(\partial h-H_{k-1})$ of each $k$-handle is a solid torus $T'\subset \partial H_2$ as in the proof of Proposition~\ref{core}, whose tight contact structure is uniquely determined rel boundary by the dividing curves. When $k=1$, $\partial_+ h=S^2\times I$ again has a unique tight contact structure $\xi$ rel boundary, so we can again isotope $f$ to make it a contactomorphism on $\partial H_k$. Now, however, we have a choice of isotopy: If $\gamma\subset S^2\times I$ is an arc transverse to $\xi$ and isotopic to $\{p\}\times I$ rel boundary, there is an obvious diffeomorphism $\psi$ of $S^2\times I$ rel boundary that fixes $\gamma$ but puts a $2\pi$-twist in its normal bundle. (Rotate $S^2\times\{t\}$ by $2\pi t$ about $p$.) We can arrange $\psi$ to be a contactomorphism (e.g.\ using a local model for the transverse arc $\gamma$ and uniqueness on the remaining ball). Since $\pi_1(\SO(3))=\Z/2$, each even power $\psi^{2m}$ has an isotopy $\Psi$ rel boundary (but moving $\gamma$) to the identity. However, there can be no such isotopy through contactomorphisms. To see this, fix a nowhere-zero vector field $v$ on $\xi$ (which can be done since $c_1(\xi)=0$). Then any contact isotopy of $\gamma$ is covered by $v|\gamma$ in an obvious way, so that if $\gamma$ is returned to its original position then so is $v|\gamma$. However, $\psi^{2m}$ changes $v|\gamma$ by $2m$ twists, so it cannot be obtained by contact isotopy rel $\partial$ from the identity. We can interpret this $2m$ as the Chern number, relative to $v$ on the boundary, of $\Psi^*(\xi)$ over $\gamma\times I\subset S^2\times I\times I$. We have now arranged $f|\partial H_k$ to be a contactomorphism as required, and for $k=1$ we can isotopically change our choice of such $f$ in an essential way by composing with even powers of $\psi$ on the 1-handles. (We take the isotopy to be ambient in $X$, but supported in the original image of $f$.)

To complete the induction and the proof, it remains to construct a suitable homotopy of $f^*J_X$ for $k<2$. Our isotopies of $f$ so far have not disturbed the induction hypothesis that $f^*J_X$ is homotopic to $J_H$ on $H-\int H_{k-1}$ through almost-complex structures inducing a fixed contact structure on $\partial H_{k-1}$. We restrict the homotopy to $H-\int H_k$, and wish to arrange the homotopy to fix the contact structure on $\partial H_k$. We already know that the contact structure is fixed away from the $k$-handles, and that the initial and final contact structures are equal everywhere, although the intermediate plane fields may not be contact on the $k$-handle boundaries. We arrange the contact structure to be fixed except on a ball in each $k$-handle boundary: The set of oriented planes at a given point in $\partial H_k$ can be identified with $S^2$ (by the positive normal vector, or by the identification of complex vector space structures with $\SO(4)/\U(2)$). This is simply connected, so we can modify our homotopy rel $t=0,1$ to fix the contact structure near a given point, giving the required condition when $k=0$. For the remaining case $k=1$, note that $\pi_2(S^2)=\Z$, so the given homotopy $\xi_t$ may be nontrivial on the 1-handle boundaries. The obstruction is measured by the relative Chern number for each 1-handle as in the previous paragraph. This reduces mod 2 to the relative Stiefel-Whitney number $\langle w_2(\xi_t,v),\gamma\times I\rangle$, which vanishes since $\xi_t$ stabilizes to the tangent bundle $T(S^2\times I)$, and the boundary framing on this induced by $(\xi,v)$ obviously extends. Thus, the homotopy changes $v|\gamma$ by an even number of twists, and after composing with an even power of $\psi$ on each 1-handle, we can assume the obstruction vanishes. We can now assume the homotopy fixes the contact structure away from a collection of balls for both $k=0,1$. We can arrange these balls to be disjoint from the attaching regions of the handles of indices $>k$ (a regular neighborhood of a 1-complex in a 3-manifold). We are then free to redefine the homotopy on the balls so that it fixes the contact structure there, reconstructing the last induction hypothesis.
\end{proof}

An analogous theorem holds in the context of Stein surfaces and exhausting plurisubharmonic functions, with a similar proof:

\begin{thm}\label{main2}
Let $V$ be a Stein surface with an exhausting plurisubharmonic Morse function $\varphi$, and let $W=\varphi^{-1}(-\infty,a]$ be some regular sublevel set (possibly empty). Let $f\co V\to X$ be a smooth embedding into a complex surface. Suppose that the complex structure $J_X$ on $X$ pulls back to one on $V$ that is homotopic rel $W$ to the original complex structure $J_V$. Then $f$ is smoothly isotopic rel $W$ to $\hat f$, with $\hat f(V)\subset f(V)$ exhibited as a Stein surface by a plurisubharmonic function obtained by postcomposing $\varphi\circ\hat f^{-1}$ with a suitable diffeomorphism of $\R$. Furthermore, $\hat f$ restricts to a contactomorphism on all level sets of $\varphi$ avoiding a preassigned neighborhood of the critical values $>a$.
\end{thm}

The isotopy may not be ambient. Consider, for example, $(V,W)=(\C^2,\emptyset)$ embedded in $X=\C\times \C$ as the complement of the closed ray $[0,\infty)\times \{0\}$. Any ambient isotopy sends this to the complement of some ray, which is not a domain of holomorphy since any holomorphic function on it extends over $\C^2$.

\begin{proof} Let $W=V_{-1}\subset V_0\subset V_1\subset\dots$ be regular sublevel sets exhausting $V$, with at most one critical level in each $V_k-V_{k-1}$. Inductively assume that $\varphi\circ f^{-1}$ is plurisubharmonic on (hence near) $f(V_{k-1})$, $f$ is a contactomorphism on $\partial V_{k-1}$ and on the specified level sets in $V_{k-1}$, and $f^*J_X$ is homotopic to $J_V$ on $V-\int V_{k-1}$, fixing the contact structure on $\partial V_{k-1}$. Following the previous proof, we will extend these hypotheses to $V_k$ by isotoping $f$ rel $V_{k-1}$ and postcomposing $\varphi$ with a suitable function, so the theorem follows. First, use Gray's Theorem on a collar of $V_{k-1}$ to extend the hypotheses over a larger sublevel set, including all specified levels below the given critical value. Now we can easily work rel $V_{k-1}$. Use the descending disks of the (finitely many) critical points in $V_k-V_{k-1}$ to construct totally real, analytic core disks in $X$ as before. Apply Eliashberg's lemma as before, then use standard models from Morse Theory to isotope $f$ and postcompose so that the newly constructed Morse function agrees with  $\varphi\circ f^{-1}$. Since the 2-handles of $V_k-V_{k-1}$ can be prechosen to lie in preassigned neighborhoods of the critical points, we can assume their attaching regions lie in standard neighborhoods of their attaching circles, so the rest of the proof works as before until the last step: We did not recover the homotopy of $f^*J_X$ over 2-handles, so the resulting homotopy fixes the contact structure only outside a collection of balls and solid tori (the co-attaching regions of the 2-handles). However, we can still assume these are disjoint from subsequent attaching regions (that is, their ascending trajectories never hit critical points), so we can redefine the homotopy as before.
\end{proof}

As explained in \cite{yfest}, Theorem~\ref{main} shows that every Stein surface becomes the interior of a Stein handlebody after suitable deformation.

\begin{cor}\label{all} \cite{yfest} Every Stein surface $V$ is isotopic within itself to the interior of a properly embedded Stein handlebody $H$ such that $V-\int H$ is diffeomorphic to $[0,1)\times \partial H$.
\end{cor}

\begin{proof} By Lemma~A.2 of \cite{yfest}, $V$ is diffeomorphic to the interior of a Stein handlebody $H$, preserving the complex structures up to homotopy. (Work is required in order to convert a Morse function with infinitely many critical values into an infinite handlebody with only three layers of handles, preserving the framing condition for the 2-handles.) Properly isotope $H$ into its own interior in the obvious way, then identify the latter with $V$. Theorem~\ref{main} completes the proof.
\end{proof}

\begin{proof}[Proof of Theorem~\ref{smooth}]
The result follows immediately from Theorem~\ref{main2}. Alternatively, further homotopy of the almost-complex structure allows us to identify the model Stein surface as the interior of a Stein handlebody, using  Lemma~A.2 of \cite{yfest} if there are infinitely many critical points. Then Theorem~\ref{main} completes the proof as for the previous corollary. The latter approach uses the nontrivial Lemma~A.2, but has the advantage of exhibiting the resulting Stein open subset as the interior of a Stein handlebody properly embedded in a simple way in $U$.
\end{proof}


\end{document}